\documentclass[a4paper]{amsart}
\usepackage{amsmath}
\usepackage{tikz}
\usepackage{skull}
 
\definecolor{opblue}{RGB}{58,110,225}
\definecolor{opred}{RGB}{210,55,60}
\newcommand{\corolla}[3]{%
  \tikz[baseline={(0,-0.16)},scale=0.42,line width=1.15pt,line cap=round]{%
    \pgfmathsetmacro{\hw}{0.30+0.08*(#1-2)}%
    \ifnum#1=1
      \draw[#2] (0,0) -- (0,0.62);%
    \else
      \foreach \i in {1,...,#1}{%
        \pgfmathsetmacro{\x}{-\hw+2*\hw*(\i-1)/(#1-1)}%
        \draw[#2] (0,0) -- (\x,0.62);}%
    \fi
    \draw[#3] (0,0) -- (0,-0.5);}}
\newcommand{\corollagen}[3]{%
  \tikz[baseline={(0,-0.16)},scale=0.42,line width=1.15pt,line cap=round]{%
    \draw[#2] (0,0) -- (-0.55,0.62);
    \draw[#2] (0,0) -- (0.55,0.62);
    \fill[#2] (-0.16,0.5) circle (0.045) (0,0.5) circle (0.045) (0.16,0.5) circle (0.045);
    \draw[#3] (0,0) -- (0,-0.5);
    \node[above,inner sep=1pt] at (0,0.62) {\scriptsize$#1$};}}

\usepackage{tikz}
\usepackage{tikz-cd}
\usetikzlibrary{arrows}
\tikzcdset{arrow style=tikz,
diagrams={>={Stealth[round,length=4pt,width=6pt,inset=3pt]}}}
\usepackage{graphicx}
\usepackage[all,cmtip]{xy}

\makeatletter
\newcommand*{\tightdisplaymath}{\abovedisplayskip\z@\belowdisplayskip\z@}

\usepackage{calligra}

\usepackage[colorlinks,allcolors=red]{hyperref}

\makeatletter
\newcommand\RSloop{\@ifnextchar\bgroup\RSloopa\RSloopb}
\makeatother
\newcommand\RSloopa[1]{\bgroup\RSloop#1\relax\egroup\RSloop}
\newcommand\RSloopb[1]
  {\ifx\relax#1
   \else
     \ifcsname RS:#1\endcsname
       \csname RS:#1\endcsname
     \else
       \GenericError{(RS)}{RS Error: operator #1 undefined}{}{}
     \fi
   \expandafter\RSloop
   \fi
  }

\newcommand{\A}{\mathscr{A}} 
\newcommand{\B}{\mathscr{B}} 
\newcommand{\Ff}{\mathscr{F}} 
\newcommand{\Gf}{\mathscr{G}} 
\newcommand{\Ain}{\mbox{A}_{\infty}} 
\newcommand{\Id}{\mbox{Id}}
\newcommand{\T}{\mathfrak{t}}

\newcommand{\Fun}{\mbox{Fun}}
\newcommand\restr[2]{{
  \left.\kern-\nulldelimiterspace 
  #1 
  \littletaller 
   \right|_{#2} 
  }}

\usepackage{graphicx}
\usepackage{amsfonts}
\usepackage{amsmath}
\usepackage{amsthm}
\usepackage{amssymb}
\usepackage[all]{xy}
\usepackage{hyperref}
\usepackage{aliascnt}
\usepackage{multicol}
\usepackage{enumitem}
\usepackage{comment}
\usepackage[mathscr]{euscript}

\theoremstyle{definition}
\newtheorem{defn}{Definition}[section]
\newtheorem{exmp}{Example}[section]
\newtheorem{rem}{Remark}[section]
\newtheorem{thm}{Theorem}[section]
\newtheorem{prop}{Proposition}[section]
\newtheorem{lemma}{Lemma}[section]
\newtheorem{cor}{Corollary}[section]
\newtheorem*{thmA}{Theorem A}
\newtheorem*{propB}{Proposition B}

\title{The (right) homotopy between two A$_{\infty}$-morphisms}

\author{Mattia Ornaghi}

\thanks{The author was supported by ERC Advanced Grant-101095900-TriCatApp.}
\date{\today}

\address{\parbox{0.9\textwidth}{Universit\`a degli Studi di Milano\\
Dipartimento di Matematica\\
Via Cesare Saldini 50, 20133 Milano, Italy}}

\email{mattia12.ornaghi@gmail.com}

\begin{document}
\maketitle

\begin{abstract}
Let $f$, $g\colon A\to B$ be two $\Ain$-morphisms between two $\Ain$-algebras (linear over a commutative ring $\mathbb{K}$).\ 
In this paper, we compare different notions of (right) homotopy between $f$ and $g$.\ 
We note that the $\Ain$-algebra $B\otimes R$, where $R:=\mathbb{K}[t,dt]$, is a model for the interval $[0,1]\times B$ in the sense of Fukaya-Oh-Ohta-Ono, and we give a short alternative proof, in the unfiltered case, of the existence of comparison morphisms between such models.\ 
We use it to prove that two $\Ain$-morphisms are homotopic \`a la Seidel if and only if, considered as morphisms of 2-colored DG operads, they are elementary homotopic \`a la Markl-Shnider-Stasheff through a homotopy which is constant on the two monochromatic parts.\ 
We also provide an example of two $\Ain$-morphisms which are elementary homotopic but not weakly equivalent (and so homotopic \`a la Seidel).\
Moreover, over a field, we observe that the models for the interval are good path objects in the category of $\Ain$-algebras with (non strict) $\Ain$-morphisms, and we explain why weakly equivalent $\Ain$-morphisms need not be homotopic.\
\end{abstract}

\section{Introduction and statement of the results}

$\Ain$-algebras and $\Ain$-categories, together with the corresponding notions of morphisms between them, appear in several areas of mathematics.\ 
They were introduced in algebraic topology by Stasheff \cite{Sta}, in the study of loop spaces and homotopy associative structures; 
in symplectic geometry they are the algebraic structures underlying Lagrangian Floer theory and Fukaya categories \cite{Fuk}, \cite{Fuk93}, \cite{FOOO1}, \cite{Sei}; 
in algebraic geometry they appear, for instance, in homological mirror symmetry \cite{Kon} and as (DG or $\Ain$)-enhancements of derived categories of coherent sheaves, \cite{Gen}, \cite{BLM}; see also \cite{KS3} for their role in noncommutative geometry.\ 
There are several equivalent ways to define them: by means of the multilinear operations $m^n$ (and $f^n$ for the morphisms) satisfying the $\Ain$-equations, see e.g.\ \cite{Kel}, \cite[(1.2)-(1.6)]{Sei}; as DG structures (codifferentials) on the bar construction \cite[\S2]{Orn1}; as algebras over the DG operad $\Ain$ and morphisms of (2-colored) DG operads \cite{MSS}, \cite{LoVa}, \cite{LAM}.\ 
In this note we assume that the reader is familiar with these notions, as well as with the various notions of unit in this setting (strict, cohomological and homotopy units), see e.g.\ \cite[\S5]{Fuk}, \cite[Definition 1.1.4-1.1.5]{Orn2}, \cite[\S2]{Sei}.

We fix a commutative ring $\mathbb{K}$ with unit $1_{\mathbb{K}}$; all our categories are assumed to be linear over $\mathbb{K}$ and small.
The basic question is the following: given two $\Ain$-algebras (or $\Ain$-categories) $A$ and $B$ and two $\Ain$-morphisms (or $\Ain$-functors) 
\begin{equation}
f,g\colon A\to B,
\end{equation}
what does it mean that $f$ and $g$ are \emph{homotopic}?\ 
Several answers coexist in the literature:
\begin{itemize}
\item[$\bullet$] $f\sim g$ if $f-g=\mathfrak{M}^1(T)$ for a prenatural transformation $T\colon f\to g$ of degree $0$ with $T^0=0$, cf. Seidel \cite[(1h)]{Sei};
\item[$\bullet$] $f\approx g$ if $f$ and $g$ are isomorphic in $H^0(\Fun(A,B))$, see \cite[Definition 1.26 (i)]{COS1};
\item[$\bullet$] $f\sim_{\mathfrak{C}_B}g$ if $f$ and $g$ are homotopic via a model $\mathfrak{C}_B$ for the interval $[0,1]\times B$, in the sense of Fukaya-Oh-Ohta-Ono \cite[\S4.2.1]{FOOO1};
\item[$\bullet$] $f\asymp g$ if $f$ and $g$, seen as morphisms of 2-colored DG operads, are elementary homotopic in the sense of Markl-Shnider-Stasheff \cite[Definition 3.121]{MSS}.\end{itemize}
It is well known that, if $B$ is strictly unital, $f\sim g$ implies $f\approx g$, while the converse is false (Example \ref{ex:fukaya}, due to Fukaya).\ 
Our aim is to compare these notions.

All of them are \emph{right homotopies}: one looks for a (kind of) ``{path object}'' $\text{Path}(B)$ such that $f$ and $g$ are homotopic if there is $\mathscr{H}$ fitting the diagram
\begin{equation}
\xymatrix{
&B\\
A\ar[r]^-{\mathscr{H}}\ar[ur]^{f}\ar[dr]_{g}&\text{Path}(B)\ar[u]_{p_0}\ar[d]^{p_1}\\
&B
}
\end{equation}
with $\text{Path}(B)=I\otimes B$ for $\sim$ (Section \ref{sec:homotopies}), $\text{Path}(B)=\mathfrak{C}_B$ in \cite{FOOO1}, and $\text{Path}(B)=P(B)\subseteq\Fun([1],B)$ for $\approx$, see Remark \ref{diag:approx}.\ 
If we consider the category of (strictly unital) $\Ain$-categories with (strictly unital) strict $\Ain$-functors, it has a model structure \cite{Orn4}.\ 
For $\Ain$-algebras and (non strict) $\Ain$-morphisms the situation is different, and it is the one relevant for the relation $\sim$.\ 
The category $\mbox{Alg}_{\infty}$ of $\Ain$-algebras has finite products but, in general, no coproducts \cite[\S 4.1]{Val}.\ 
Hence in $\mbox{Alg}_{\infty}$ there are no cylinder objects (they would factor the codiagonal $A\sqcup A\to A$), while path objects (factorizations of the diagonal $B\to B\times B$) do exist: over a field, $\mbox{Alg}_{\infty}$ is a category of fibrant objects in the sense of Brown \cite{Bro}, see \cite[Th\'eor\`eme 1.3.3.1]{LH} and \cite[\S 4]{Val}.\ 
This is why right homotopies are the natural notion for $\Ain$-algebras.

\subsection*{Main results}
Let $f,g\colon A\to B$ be two $\Ain$-morphisms.\ Considered as morphisms of 2-colored DG operads, $f$ and $g$ are elementary homotopic if there is a morphism of 2-colored DG operads $\tau\colon A^2_\infty\to Hom(A,B)\otimes R$, where $R:=\mathbb{K}[t,dt]$, which gives $f$ for $t=0$ and $g$ for $t=1$.\ 
We write $f\asymp^{c}g$ if moreover $\tau$ can be chosen constant on the two monochromatic parts of $A^2_\infty$, i.e.\ equal to $m_A\otimes 1_R$ and $m_B\otimes 1_R$ there.

\begin{thmA}[Theorem \ref{thm:MSS}, Corollary \ref{cor:equiv}, Example \ref{ex:MSS}]
Suppose that $\mathbb{Q}\subset\mathbb{K}$, let $A$, $B$ be two (non unital) $\Ain$-algebras and let $f,g\colon A\to B$ be two $\Ain$-morphisms.\ Then
\begin{align*}
f\sim g\quad\Longleftrightarrow\quad f\asymp^{c}g .
\end{align*}
In particular $f\sim g$ implies $f\asymp g$, but the converse is false in general: without the constancy condition, elementary homotopy is strictly coarser than the homotopy of Seidel.\ In fact, we provide an example (Example (\ref{ex:MSS})) showing that $f\asymp g$ does not imply $f\approx g$ (and so $f\sim g$).
\end{thmA}

Both implications rely on comparison morphisms between models for the interval.\ 
The implication $f\asymp^{c}g\Rightarrow f\sim g$ uses that $B\otimes R$ is a model for the interval $[0,1]\times B$ (Theorem \ref{thm:R}), and the independence of the homotopy relation from the choice of the model (Theorem \ref{thm:indep}), which is due to Fukaya-Oh-Ohta-Ono \cite[Theorem 4.2.34 and Lemma 4.2.36]{FOOO1}.\ 
In the unfiltered case we give a short proof of it (Proposition \ref{prop:comparison}): the induction on the number filtration of \cite[\S 4.4]{FOOO1} is replaced by the contractibility of a complex of the form $Hom(Q,J)$, where $Q$ is the cokernel of a tensor power of $\text{Incl}$ and $J$ is the kernel of $\text{Eval}_0\times\text{Eval}_1$.\ 
The implication $f\sim g\Rightarrow f\asymp^{c}g$ uses a refinement of the same argument (Lemma \ref{lem:finite}): the comparison morphism $I\otimes B\to B\otimes R$ can be chosen with components in $Hom\big((I\otimes B)^{\otimes k},B\big)\otimes R\subset Hom\big((I\otimes B)^{\otimes k},B\otimes R\big)$.

Over a field, the independence of the model is also a consequence of the (almost) model structure on $\mbox{Alg}_{\infty}$ of Lef\`evre-Hasegawa \cite{LH} and Vallette \cite{Val}.

\begin{propB}[Propositions \ref{prop:path} and \ref{prop:indep}]
Let $\mathbb{K}$ be a field.\ Every model for the interval $[0,1]\times B$ is a very good path object for $B$ in $\mbox{Alg}_{\infty}$.\ 
If two $\Ain$-morphisms $f,g\colon A\to B$ are right homotopic through some path object of $B$, then they are right homotopic through every good path object of $B$.\ 
In particular $f\sim_{\mathfrak{C}_B}g$ if and only if $f\sim_{\mathfrak{C}'_B}g$ for any two models $\mathfrak{C}_B$, $\mathfrak{C}'_B$ for the interval.
\end{propB}

Proposition B gives a model-categorical interpretation of the models for the interval, and of the relation $\sim$ as the right homotopy relation of $\mbox{Alg}_{\infty}$ (Corollary \ref{cor:sim-right}).\ 
Finally, the same point of view explains why $f\approx g$ does not imply $f\sim g$ (Remark \ref{rem:notgood}): $I\otimes B$ is a good path object in $\mbox{Alg}_{\infty}$ (whose weak equivalences are the quasi-isomorphisms), but it is still a path object, which is not good, in the Tabuada model structure on DG categories (whose weak equivalences are the quasi-equivalences), where a good path object is $P(B)$, see Remark \ref{diag:approx}.

\subsection*{$\skull$ Disclaimer} During the writing of these notes we were informed by Guillaume Laplante-Anfossi, Thibaut Mazuir and Bingyu Zhang that there is another characterization of the (Seidel) homotopy making use of a particular cylinder object.\ Using this characterization part of the results of this paper follows by abstract nonsense at the cost of checking some technical details and assuming some extra hypotheses.\ On the other hand, this strategy is very different in spirit from ours.

\tableofcontents

\subsection*{Notation and conventions}
We work over a commutative ring $\mathbb{K}$.\ All our categories are assumed to be $\mathbb{K}$-linear and small.\ 
Given a category $\mathscr{C}$, we denote by $\text{Ob}(\mathscr{C})$ its set of objects and by $\mathscr{C}(c_1,c_2)$ the hom-space between $c_1,c_2\in\text{Ob}(\mathscr{C})$.\ 
Given a graded $\mathbb{K}$-module $M=\oplus_{i\in\mathbb{Z}}M^i$ and an integer $n$, the shifted graded module $M[n]$ is defined by $M[n]^i:=M^{n+i}$.\ Complexes are cohomological, namely $d_M\colon M^i\to M^{i+1}$.\ If $x\in M^i$ is homogeneous, we set $\deg(x):=i$ and $\deg'(x):=\deg(x)-1$ (the \emph{shifted degree}).\ 
A $\mathbb{K}$-linear map $\phi\colon M\to N$ between graded $\mathbb{K}$-modules has degree $k$ if $\phi(M^i)\subseteq N^{i+k}$ for every $i$, i.e.\ if it is a degree $0$ map $M\to N[k]$.

\section{Background and notation}\label{sec:background}

In this section we briefly recall the basic notions about $\Ain$-categories and $\Ain$-algebras.\ 
We adopt the sign conventions of Seidel's book \cite{Sei}, as in \cite{Orn1}.\ They are different from, but equivalent to, the conventions of Lef\`evre-Hasegawa \cite{LH} used in \cite{COS1} and \cite{COS2}, see \cite[(1.1)]{COS2}; when we quote results from \cite{COS1} and \cite{COS2}, we implicitly use this identification.
\begin{defn}\label{defn:Ainfty}
An \emph{$\Ain$-category} $\A$ is a graded $\mathbb{K}$-linear quiver equipped with $\mathbb{K}$-linear maps
\begin{align*}
m^n_{\A}\colon \A(x_{n-1},x_n)\otimes...\otimes \A(x_0,x_1)\to \A(x_0,x_n)[2-n],
\end{align*}
for every positive integer $n$ and all $x_0,...,x_n\in\text{Ob}(\A)$, such that
\begin{align}\label{eq:Ainf-cat}
\sum^n_{m=1}\sum^{n-m}_{d=0}(-1)^{\dagger_d}\, m_{\A}^{n-m+1}\big(a_n,...,a_{d+m+1},m_{\A}^{m}(a_{d+m},...,a_{d+1}),a_d,...,a_1\big)=0,
\end{align}
where $\dagger_d:=\deg'(a_d)+...+\deg'(a_1)=\deg(a_d)+...+\deg(a_1)-d$.\ 
We say that $\A$ is \emph{strictly unital} if for every object $x$ there is $1_x\in \A(x,x)^0$ such that $m^1_{\A}(1_x)=0$, 
$m^2_{\A}(a,1_x)=a$ and $m^2_{\A}(1_y,a)=(-1)^{\deg(a)}a$ for every $y\in\text{Ob}(\A)$ and every $a\in \A(x,y)$, and $m^n_{\A}(a_n,...,a_1)=0$ whenever $n\ge3$ and some $a_i$ is a unit $1_x$.\ 
An \emph{$\Ain$-algebra} is an $\Ain$-category with one object.
\end{defn}
\begin{defn}
An $\Ain$-functor $\Ff\colon \A\to \B$ is given by a map of sets $\Ff^0\colon\mbox{Ob}(\A)\to\mbox{Ob}(\B)$ together with $\mathbb{K}$-linear maps
\begin{align*}
\Ff^n\colon \A(x_{n-1},x_n)\otimes...\otimes \A(x_0,x_1)\to \B\big(\Ff^0(x_0),\Ff^0(x_n)\big)[1-n],
\end{align*}
for every positive integer $n$ and all $x_0,...,x_n\in\text{Ob}(\A)$, such that
\begin{equation}\label{eq:Ainf-fun}
\begin{split}
\sum_{r\ge1}\ \sum_{\substack{s_1+...+s_r=n\\ s_1,...,s_r\ge1}}&m^{r}_{\B}\big(\Ff^{s_r}(a_n,...,a_{n-s_r+1}),...,\Ff^{s_1}(a_{s_1},...,a_1)\big)\\
&=\sum_{m=1}^n\sum_{k=0}^{n-m}(-1)^{\dagger_k}\, \Ff^{n-m+1}\big(a_n,...,a_{k+m+1},m^{m}_{\A}(a_{k+m},...,a_{k+1}),a_k,...,a_1\big).
\end{split}
\end{equation}
We denote by $(\star_n)$ the equation (\ref{eq:Ainf-fun}) for a fixed $n$.\ If $\A$ and $\B$ are strictly unital, the $\Ain$-functor $\Ff$ is \emph{strictly unital} if $\Ff^1(1_x)=1_{\Ff^0(x)}$ for every $x\in\text{Ob}(\A)$, and $\Ff^n(a_n,...,a_1)=0$ whenever $n\ge2$ and some $a_i$ is a unit.\ 
An \emph{$\Ain$-morphism} is an $\Ain$-functor between two $\Ain$-algebras.\ An $\Ain$-functor (resp.\ $\Ain$-morphism) $\Ff$ is \emph{strict} if $\Ff^n=0$ for every $n\ge2$.
\end{defn}
\begin{defn}\label{defn:DG}
A \emph{DG category} is a strictly unital $\Ain$-category $\A$ such that $m^n_{\A}=0$ for every $n\ge3$; a \emph{DG algebra} is a DG category with one object.\ 
A strict and strictly unital $\Ain$-functor between two DG categories is called a \emph{DG functor}.
\end{defn}
Let $\A$ and $\B$ be two (non unital) $\Ain$-categories and let $\Ff,\Gf\colon\A\to\B$ be two (non unital) $\Ain$-functors.\ 
The cohomology $H(\A)$ is a non-unital graded category, with composition $[a_2]\cdot[a_1]:=(-1)^{\deg(a_1)}[m^2_{\A}(a_2,a_1)]$, and $\Ff$ induces a non-unital graded functor $H(\Ff)\colon H(\A)\to H(\B)$.\ 
We say that $\A$ is \emph{cohomologically unital} if $H(\A)$ has identities.\ 
If $\B$ is cohomologically unital, $\Ff$ is a \emph{quasi-equivalence} if $\Ff^1$ is a quasi-isomorphism on every hom-complex and $H(\Ff)$ is essentially surjective, cf.\ \cite[Definition 1.13]{COS2}; in this case $\A$ is cohomologically unital and $H(\Ff)$ is an equivalence of graded categories.\ 
We say that an $\Ain$-morphism $f$ is a \emph{quasi-isomorphism} if $f^1$ induces an isomorphism in cohomology.\\

Recall that the category of DG categories has a model structure whose weak equivalences are the quasi-equivalences \cite{Tab1}.\ 
On the other hand, the category of $\Ain$-categories does not admit a model structure with all limits and colimits, see \cite[\S1.5]{COS1}; for $\Ain$-algebras over a field see however Theorem \ref{thm:LH}.
It is well known that the $\Ain$-functors $\A\to\B$ are the objects of an $\Ain$-category $\Fun(\A,\B)$, see \cite{Lyu} and \cite{Sei}, whose morphisms are the prenatural transformations:

\begin{defn}\label{prenat}
A \emph{prenatural transformation} $T\colon\Ff\to\Gf$ of degree $p$ is a family $T=\{T^n\}_{n\ge 0}$, where:
\begin{itemize}
\item $T^0(x)\in \B\big(\Ff^0(x),\Gf^0(x)\big)^p$ for every $x\in\mbox{Ob}(\A)$;
\item for every $n\ge1$ and all $x_0,...,x_n\in \mbox{Ob}(\A)$,
\begin{align*}
T^n\colon \A(x_{n-1},x_n)\otimes...\otimes \A(x_0,x_1)\to \B\big(\Ff^0(x_0),\Gf^0(x_n) \big)[p-n]
\end{align*}
is a $\mathbb{K}$-linear map of degree $0$, i.e.\ $T^n$ has degree $p-n$.
\end{itemize}
\end{defn}

The $\Ain$-category of $\Ain$-functors from $\A$ to $\B$ is denoted by
\begin{align}\label{Funzo}
\big(\Fun(\A,\B),\mathfrak{M}^n\big).
\end{align}
Given $T\colon\Ff\to\Gf$, the prenatural transformation $\mathfrak{M}^1(T)\colon\Ff\to\Gf$ is defined by $\mathfrak{M}^1(T)^0(x):=m^1_{\B}(T^0(x))$ and, for $n\ge1$:
\begin{align*}
\mathfrak{M}^1(T)^n(a_n,...,a_1):=\ &\sum_{r\ge1}\ \sum_{i=1}^{r}\ \sum_{s_1+...+s_r=n}(-1)^{\diamond_i}\\
&\qquad m^r_{\B}\big(\Gf^{s_r}(a_n,...),...,\Gf^{s_{i+1}}(...),T^{s_i}(...),\Ff^{s_{i-1}}(...),...,\Ff^{s_1}(...,a_1)\big)\\
&-\sum_{m=1}^{n}\ \sum_{k=0}^{n-m}(-1)^{p-1+\dagger_k}\,T^{n-m+1}\big(a_n,...,a_{k+m+1},m^{m}_{\A}(a_{k+m},...,a_{k+1}),a_k,...,a_1\big),
\end{align*}
where in the first sum $s_i\ge0$ and $s_j\ge1$ for $j\neq i$ and $\diamond_i:=(p-1)\,\dagger_{s_1+...+s_{i-1}}$; see \cite[(1.8)]{Sei} and \cite[Definition 1.3.2]{Orn1}.\ 
We say that a prenatural transformation $T$ is \emph{natural} if $\mathfrak{M}^1(T)=0$.\ 
There is a strictly unital version of Definition \ref{prenat}, cf.\ \cite[Remark 1.24]{COS1} and \cite[Definition 1.17]{COS2}, and $\Fun(\A,\B)$ has the same structure as $\B$, e.g. if $\B$ is DG (resp.\ strictly unital), then so is (\ref{Funzo}).

We conclude this section by introducing a notation of Fukaya-Oh-Ohta-Ono \cite[p.~218]{FOOO1}, which will be useful in the following sections.\ 
Let $A$ and $C$ be two $\Ain$-algebras.\ We denote by $Hom(A^{\otimes n},C)$ the complex of graded $\mathbb{K}$-linear maps.\ 
The differential is
\begin{align*}
d_{Hom(A^{\otimes n},C)}(\phi)(a_n,...,a_1):=\ &m^1_C\big(\phi(a_n,...,a_1)\big)\\
&-(-1)^{\deg'(\phi)}\sum_{j=1}^n(-1)^{\dagger_{j-1}}\phi\big(a_n,...,a_{j+1},m^1_A(a_j),a_{j-1},...,a_1\big).
\end{align*}
Let $f=(f^n)_{n\ge1}$ be a family of maps $f^n\in Hom(A^{\otimes n},C)$ of degree $1-n$.\ 
Isolating the terms containing $m^1$, the condition $(\star_n)$ can be rewritten as
\begin{equation}\label{eq:starn-d}
\begin{split}
d_{Hom(A^{\otimes n},C)}(f^n)\big(a_n,...,a_1\big)=&\sum_{j=2}^{n}\ \sum_{k=0}^{n-j}(-1)^{\dagger_k} f^{n-j+1}\big(a_n,...,a_{k+j+1},m^j_A(a_{k+j},...,a_{k+1}),a_k,...,a_1\big)\\
&-\sum_{r=2}^{n}\ \sum_{\substack{s_1+...+s_r=n\\ s_1,...,s_r\ge1}} m^r_{C}\big(f^{s_r}(a_n,...,a_{n-s_r+1}),...,f^{s_1}(a_{s_1},...,a_1)\big).
\end{split}
\end{equation}
We denote by $\text{Err}_n(f)$ the right hand side of (\ref{eq:starn-d}).\ Note that it depends only on $f^1,...,f^{n-1}$.\ 
If $f^1,...,f^{n-1}$ satisfy $(\star_k)$ for every $k<n$, then $\text{Err}_n(f)$ is a cycle of $Hom(A^{\otimes n},C)$.\ 
In particular, if $f$ is an $\Ain$-morphism, then $\text{Err}_n(f)=d_{Hom(A^{\otimes n},C)}(f^n)$ for every $n$.

\section{Homotopies}\label{sec:homotopies}

In this section we recall the notions of homotopy between $\Ain$-functors.\ 
In particular we define when two $\Ain$-functors (or $\Ain$-morphisms) are \emph{homotopic} and \emph{weakly equivalent} (Subsection \ref{subsec:homotwekequi}), \emph{homotopic via an interval} (Subsection \ref{susec:interval}), and \emph{elementary homotopic}, as morphisms of 2-colored DG operads (Subsection \ref{subsec:elemhomot}).\ 
Finally we introduce elementary homotopies \emph{constant on the monochromatic parts} which is new in the literature.

\subsection{Homotopic and weakly equivalent $\Ain$-functors}\label{subsec:homotwekequi}
Let $\Ff,\Gf\colon\A\to\B$ be two non unital $\Ain$-functors between two non unital $\Ain$-categories.
\begin{defn}\label{homto}
We say that two $\Ain$-functors $\Ff,\Gf\colon \A\to \B$ are \emph{homotopic}, in formula $\Ff\sim \Gf$, if: 
\begin{itemize}
\item $\Ff^0=\Gf^0$ as maps of sets,
\item there exists a prenatural transformation $T\colon\Ff\to\Gf$ of degree $0$ such that 
\begin{align*}
\Ff-\Gf=\mathfrak{M}^1(T)\qquad\text{and}\qquad T^0=0.
\end{align*}
\end{itemize}
\end{defn}
There is an equivalent definition inspired by rational homotopy theory, cf.\ \cite[Remark 1.11]{Sei} and \cite[Chapter X]{GrMo}.\ 
We say that $\Ff$ is \emph{homotopic} to $\Gf$ if there exists an $\Ain$-functor $\mathscr{H}\colon \A\to I\otimes \B$ fitting the diagram:
\begin{align}\label{diag:pathint}
\xymatrix{
&\B\\
\A\ar[r]^-{\mathscr{H}}\ar[ur]^{\Ff}\ar[dr]_{\Gf}&I\otimes \B\ar[u]_{p_0\otimes\tiny\mbox{Id}_{\B}}\ar[d]^{p_1\otimes\tiny\mbox{Id}_{\B}}\\
&\B
}
\end{align}
Here $I$ is the DG algebra 
\begin{align}\label{cat:I}
I:= \mathbb{K}u_0 \oplus \mathbb{K}u_1\oplus \mathbb{K}h,\qquad \deg(u_0)=\deg(u_1)=0,\quad \deg(h)=1,
\end{align}
with differential $du_0=h$, $du_1=-h$, $dh=0$, and product given by:
\begin{align*}
u_0u_0=u_0,\quad u_1u_1=u_1, \quad u_0u_1=u_1u_0=h u_1=u_0h=hh=0,\quad  u_1h=hu_0=h.
\end{align*}
The unit of $I$ is $u_0+u_1$.\ 
Note that $I\otimes\B$ is the tensor product of a DG algebra and an $\Ain$-category, which is a particular case of the tensor product of two $\Ain$-categories, see \cite[(5)]{Orn3}.\ 
We have two quasi-isomorphisms of DG algebras
\begin{align*}
p^1_0:I&\to \mathbb{K}\\
xu_0+yu_1+zh&\mapsto x\cdot 1_{\mathbb{K}}
\end{align*}
and
\begin{align*}
p^1_1:I&\to \mathbb{K}\\
xu_0+yu_1+zh&\mapsto y\cdot 1_{\mathbb{K}}.
\end{align*}
Moreover, we have a morphism of DG algebras 
\begin{align*}
i\colon \mathbb{K} &\to I\\
k 1_{\mathbb{K}}&\mapsto k(u_0+u_1)
\end{align*}
providing a strict $\Ain$-functor $i\otimes\Id_{\B}\colon \B \to I\otimes \B$.\ Similarly the two projections $p^1_k\otimes\Id_{\B}\colon I\otimes\B\to\B$ ($k=0,1$) are strict $\Ain$-functors.

\begin{rem}
Definition \ref{homto} applies to $\Ain$-algebras; in this case the condition $\Ff^0=\Gf^0$ is automatic.
\end{rem}

Now we suppose that $\B$ is strictly unital, so that $\Fun(\A,\B)$ is strictly unital.\ If $\B$ is only unital, one uses the units of $H^0(\Fun(\A,\B))$, cf.\ \cite[Definition 1.18]{COS2}.

\begin{defn}
Two (non unital) $\Ain$-functors $\Ff,\Gf\colon\A\to\B$ are \emph{weakly equivalent}, in formula $\Ff\approx\Gf$, if they are isomorphic in $H^0(\Fun(\A,\B))$.\ Namely, there exist two natural transformations $H\colon\Ff\to \Gf$ and $H'\colon\Gf\to \Ff$ of degree $0$ and two prenatural transformations $W\colon\Ff\to \Ff$ and $W'\colon\Gf\to \Gf$ of degree $-1$ such that: 
\begin{align}\label{appross}
\mathfrak{M}^2(H',H)=\Id_{\Ff}+\mathfrak{M}^1(W)\mbox{ and } \mathfrak{M}^2(H,H')=\Id_{\Gf}+\mathfrak{M}^1(W').
\end{align}
Here $\Id_{\Ff}$ is the natural transformation with $(\Id_{\Ff})^0(x)=1_{\Ff^0(x)}$ for every $x\in\text{Ob}(\A)$ and $(\Id_{\Ff})^n=0$ for $n\ge1$, and similarly for $\Id_{\Gf}$.
\end{defn}

Now we assume that $\A$ and $\B$ are strictly unital, and we denote by $\Fun$ the strictly unital version of the $\Ain$-category of $\Ain$-functors, whose objects are the strictly unital $\Ain$-functors and whose morphisms are the strictly unital prenatural transformations (it is denoted by $\Fun^u_{\infty}$ in \cite{Orn2}).

We denote by $[0]=\mathbb{K}$ the DG category with one object, we have an isomorphism $\Fun([0], \B)\cong\B$.\ 
We denote by $[1]$ the DG category with two objects $0,1$, whose only non-identity morphism is a closed morphism of degree zero
\begin{align*}
\xymatrix{
0\ar[r]^{j_{01}}&1,
}
\end{align*} 
see \cite[Definition 3.1]{Orn5}, and by $\overline{[1]}$ the DG category with two objects $0,1$ and $\overline{[1]}(a,b)=\mathbb{K}\cdot j_{ab}$ concentrated in degree $0$, for $a,b\in\{0,1\}$, where $j_{00}=\Id_0$, $j_{11}=\Id_1$ and
\[
\xymatrix{
0\ar@/^/[r]^{j_{01}}&1\ar@/^/[l]^{j_{10}}
}
\]
with $j_{01}\circ j_{10}=\Id_1$ and $j_{10}\circ j_{01}=\Id_0$.\ 
We have DG functors $i_0,i_1\colon[0]\to[1]$ and $p_{01}\colon[1]\to [0]$, with $i_0(0)=0$, $i_1(0)=1$ and $p_{01}(j_{01})=\Id_0$, and similarly for $\overline{[1]}$.

Let $P(\B)\subseteq\Fun([1],\B)$ be the full $\Ain$-subcategory whose objects are the strictly unital $\Ain$-functors $\mathscr{X}\colon[1]\to\B$ such that $\mathscr{X}^1(j_{01})$ is a homotopy equivalence, i.e.\ its class is invertible in $H^0(\B)$.\ 
Composing with $i_0$, $i_1$ and $p_{01}$ we obtain three strict $\Ain$-functors\footnote{Using Seidel's notation \cite[(1e)]{Sei}, $s=\mathcal{R}_{i_0}$, $t=\mathcal{R}_{i_1}$ and $p=\mathcal{R}_{p_{01}}$; note that $\mathcal{R}_{-}$ is always a strict $\Ain$-functor.}
\begin{align*}
s\colon P(\B)\to \Fun([0],\B)\cong \B,\qquad
t\colon P(\B)\to  \Fun([0],\B)\cong\B,\qquad
p\colon \B\cong\Fun({[0]},\B)\to P(\B).
\end{align*}

If $\Ff \approx \Gf$, there exist two natural transformations $H\colon \Ff\to \Gf$ and $H'\colon \Gf\to\Ff$ satisfying (\ref{appross}).\ 
In particular, for every $x\in\text{Ob}(\A)$ we have $m^2_{\B}(H'^0(x),H^0(x))=1_{\Ff^0(x)}+m^1_{\B}(W^0(x))$ and $m^2_{\B}(H^0(x),H'^0(x))=1_{\Gf^0(x)}+m^1_{\B}(W'^0(x))$, so that $H^0(x)$ is a homotopy equivalence.\ 
Hence the $\Ain$-functor $\psi_{H}\colon \A\to\Fun([1],\B)$ of \cite[Lemma 2.8 and Corollary 2.9]{Orn2} takes values in $P(\B)$, and we have the commutative diagram:
\begin{align}\label{curzolo}
\xymatrix{
&\B&\\
\A\ar[ur]^{\Ff}\ar[dr]_{\Gf}\ar[r]^-{\psi_H}&{P(\B)}\ar[u]_s^{\simeq}\ar[d]^t_{\simeq}&\ar@{=}[ul]\ar@{=}[dl] \B\ar[l]_-{p}^-{\simeq}\\
&\B&
}
\end{align}
We recall that in diagram (\ref{curzolo}) the only non-strict $\Ain$-functors are $\Ff$, $\Gf$ and $\psi_H$.

\begin{rem}\label{diag:approx}
If $\B$ is a DG category, then $P(\B)$ is the DG category described in \cite[Lemma 4.8]{Gen}, which is the path object of $\B$ in the Tabuada model structure, see \cite[\S3]{Tab2} and \cite[\S 2.2]{CS}.\ 
Explicitly, an object of $P(\B)$ is a triple $(x,y,\varphi)$, where $\varphi\in\B(x,y)^0$ is a homotopy equivalence (the image of $j_{01}$), and a morphism $(x,y,\varphi)\to(x',y',\varphi')$ of degree $q$ is a matrix
\begin{align*}
\Bigl(\:\begin{matrix}
u & 0\\
\eta & v
\end{matrix}\:\Bigr),\qquad u\in\B(x,x')^q,\quad v\in\B(y,y')^q,\quad \eta\in\B(x,y')^{q-1},
\end{align*}
corresponding to the strictly unital prenatural transformation $T$ with $T^0(0)=u$, $T^0(1)=v$ and $T^1(j_{01})=\eta$.\ The composition and the differential are given in \cite[Definition 1.3.11]{Gen} and correspond to $\mathfrak{M}^2$ and $\mathfrak{M}^1$, cf.\ \cite[(1.8) and (1.9)]{Sei}.\ 
Restricting along the inclusion $[1]\hookrightarrow\overline{[1]}$ we get a strict $\Ain$-functor $\Fun(\overline{[1]},\B)\to P(\B)$, compatible with $s$, $t$ and $p$.
\end{rem}

\subsection{The ($\Ain$-algebra) interval $[0,1]$}\label{susec:interval}

In the previous subsection we have seen that, under suitable assumptions on $\A$ and $\B$, the relations $\Ff\sim\Gf$ and $\Ff\approx\Gf$ can be described by diagrams of the form
\begin{align}\label{diag:int}
\xymatrix{
&\B\\
\A\ar[r]^-{\mathscr{H}}\ar[ur]^{\Ff}\ar[dr]_{\Gf}&\hat{\B}\ar[u]_{\alpha_{\B}}\ar[d]^{\beta_{\B}}\\
&\B
}
\end{align}
where $\hat{\B}$, $\alpha_{\B}$ and $\beta_{\B}$ depend on the $\Ain$-category $\B$.\ 
In this subsection, assuming that $\A=A$ and $\B=B$ are $\Ain$-algebras, we recall the definition of $\Ain$-morphisms \emph{homotopic via an interval} (Definition \ref{larks}), which generalizes the situation in (\ref{diag:int}).\ 
First we need a good definition of $\hat{\B}$.

Recall that the product of two $\Ain$-algebras $A$ and $B$ is $A\times B:=A\oplus B$ with the componentwise operations
\begin{align*}
m^n_{A\times B}\big((a_n,b_n),...,(a_1,b_1)\big):=\big(m^n_A(a_n,...,a_1),m^n_B(b_n,...,b_1)\big).
\end{align*}
Note that the two projections are strict $\Ain$-morphisms.\

The next definition is the unfiltered version of \cite[Definition 4.2.1]{FOOO1}.
\begin{defn}\label{defn:model}
Let $B$ be an $\Ain$-algebra.\ A \emph{model for the interval} $[0,1]\times B$ is given by 
an $\Ain$-algebra $\mathfrak{C}_B$ and three $\Ain$-morphisms 
\begin{align}
\text{Eval}_0\colon \mathfrak{C}_B\to B,\quad \text{Eval}_1\colon \mathfrak{C}_B\to B,\quad \text{Incl}\colon B\to \mathfrak{C}_B,
\end{align}
such that:
\begin{itemize}
\item[MI1)] $\text{Eval}_0$, $\text{Eval}_1$ and $\text{Incl}$ are strict $\Ain$-morphisms. 
\item[MI2)] $\text{Eval}_0\circ \text{Incl}=\Id_B$, $\text{Eval}_1\circ \text{Incl}=\Id_B$.
\item[MI3)] $(\text{Eval}_0)^1$, $(\text{Eval}_1)^1$ and $\text{Incl}^1$ are cochain homotopy equivalences.
\item[MI4)] $\text{Eval}_0\times \text{Eval}_1\colon \mathfrak{C}_B\to B\times B$ is split surjective, i.e.\ $(\text{Eval}_0)^1\oplus (\text{Eval}_1)^1\colon \mathfrak{C}_B\to B\oplus B$ is a surjective map admitting a section as a map of graded $\mathbb{K}$-modules.\footnote{In the proof of \cite[Lemma 4.4.2]{FOOO1} (p.~219) this map is used as a \emph{split} surjection (by (4.2.2.4) and Remark 4.2.4 there), and $\text{Incl}^1$ as a split injection (which follows from MI2).}
\end{itemize}
\end{defn}
Note that Definition \ref{defn:model} concerns $\Ain$-algebras, not $\Ain$-categories.

\begin{exmp}\label{ex:I}
Let $I$ be the (non-commutative) DG algebra (\ref{cat:I}).\ The $\Ain$-algebra $I\otimes B$, together with $\text{Eval}_k:=p^1_k\otimes\Id_B$ ($k=0,1$) and $\text{Incl}:=i\otimes\Id_B$, is a model for the interval $[0,1]\times B$, over any commutative ring $\mathbb{K}$.
\end{exmp}

\begin{exmp}
If $B$ is strictly unital, the $\Ain$-category $P(B)$ of Remark \ref{diag:approx} (and similarly $\Fun(\overline{[1]},B)$) is not a model for the interval $[0,1]\times B$: in general it has more than one object, hence it is not an $\Ain$-algebra.\ It is instead a path object in the homotopy theory of DG categories, see Remark \ref{rem:notgood}.
\end{exmp}

We define the graded commutative DG algebra $R:=\mathbb{K}[t,dt]=\mathbb{K}[t]\oplus\mathbb{K}[t]\,dt$, with $\deg(t)=0$, $\deg(dt)=1$, $(dt)^2=0$ and differential $d(p(t))=p'(t)\,dt$, $d(q(t)\,dt)=0$ (in particular $d(t)=dt$).\ 
Given an $\Ain$-algebra $B$, the $\Ain$-algebra $B\otimes R$ is defined by
\begin{align*}
m^1_{B\otimes R}(b\otimes r)&:=m^1_B(b)\otimes r+(-1)^{\deg'(b)}\,b\otimes d(r),\\
m^n_{B\otimes R}(b_n\otimes r_n,...,b_1\otimes r_1)&:=(-1)^{\epsilon}\,m^n_B(b_n,...,b_1)\otimes r_1\cdots r_n,\qquad n\ge2,
\end{align*}
where $\epsilon:=\sum_{j=1}^{n}\deg(r_j)\big(\deg'(b_{j+1})+...+\deg'(b_n)\big)$.\ 
We have two morphisms of DG algebras ($k=0,1$):
\begin{align*}
\text{ev}_k\colon R&\to \mathbb{K}\\
p(t)+q(t)dt&\mapsto p(k)\cdot 1_{\mathbb{K}}.
\end{align*} 
They induce strict $\Ain$-morphisms
\begin{align*}
\text{Eval}_k:=\text{Id}_{B}\otimes\text{ev}_k\colon B\otimes R&\to B\otimes\mathbb{K}\simeq B.
\end{align*}
In the same vein, taking the morphism of DG algebras 
\begin{align*}
\text{i}\colon  \mathbb{K} &\to R\\
1_{\mathbb{K}}&\mapsto 1_{\mathbb{K}},
\end{align*} 
we have a strict $\Ain$-morphism
\begin{align*}
\text{Incl}:=\text{Id}_{B}\otimes\text{i}\colon B\otimes \mathbb{K}\simeq B&\to B\otimes R.
\end{align*}

\begin{thm}\label{thm:R}
Assume that $\mathbb{Q}\subset\mathbb{K}$.\ The $\Ain$-algebra $B\otimes R$, together with $\text{Eval}_k$ and $\text{Incl}$, is a model for the interval $[0,1]\times B$ in the sense of Definition \ref{defn:model}.
\end{thm}

\begin{proof}
The map $\text{ev}_k$ is multiplicative and unital, and it is a map of complexes because $\text{ev}_k(d\,p(t))=\text{ev}_k(p'(t)dt)=0$.\ Since $\text{ev}_k$ and $\text{i}$ are morphisms of DG algebras, $\text{Eval}_k$ and $\text{Incl}$ are strict $\Ain$-morphisms (MI1), and $\text{ev}_k\circ\text{i}=\Id_{\mathbb{K}}$ gives MI2.\ 
For MI3, since $\mathbb{Q}\subset\mathbb{K}$ we can define the degree $-1$ maps
\begin{align*}
s_k\colon R\to R,\qquad s_k(p(t)):=0,\qquad s_k(q(t)dt):=\int_k^t q(\tau)d\tau,\qquad k=0,1,
\end{align*}
and we have $d\circ s_k+s_k\circ d=\Id_R-\text{i}\circ\text{ev}_k$: on $p(t)$ both sides give $p(t)-p(k)$, on $q(t)dt$ both sides give $q(t)dt$.\ 
Hence $\text{i}$ and $\text{ev}_k$ are mutually inverse homotopy equivalences, and so are $\text{Incl}^1$ and $(\text{Eval}_k)^1$.\ 
For MI4, the $\mathbb{K}$-linear map $(a,b)\mapsto a+(b-a)t$ is a section of $(\text{ev}_0,\text{ev}_1)\colon R\to\mathbb{K}\oplus\mathbb{K}$, and tensoring with $B$ we get a section of $(\text{Eval}_0)^1\oplus(\text{Eval}_1)^1$.
\end{proof}

\begin{defn}\label{larks}
Let $\mathfrak{C}_B$ be a model for the interval $[0,1]\times B$.\ Two $\Ain$-morphisms $f,g\colon A\to B$ are \emph{homotopic via $\mathfrak{C}_B$}, in formula $f\sim_{\mathfrak{C}_B}g$, if there is an $\Ain$-morphism 
\begin{align*}
h\colon A\to \mathfrak{C}_B
\end{align*}
such that $\text{Eval}_0\circ h=f$ and $\text{Eval}_1\circ h=g$; we call $h$ a \emph{homotopy} between $f$ and $g$.
\end{defn}

\begin{thm}\label{thm:indep}
Given $\mathfrak{C}_B$ and $\mathfrak{C}'_B$ two models for the interval $[0,1]\times B$, 
we have $f\sim_{\mathfrak{C}_B}g$ if and only if $f\sim_{\mathfrak{C}'_B}g$.
\end{thm}

Theorem \ref{thm:indep} follows from the following Proposition \ref{prop:comparison}, which is the unfiltered version of \cite[Theorem 4.2.34]{FOOO1}.\ Indeed, applying it to $f=\Id_B$, we get an $\Ain$-morphism $\mathfrak{F}\colon\mathfrak{C}_B\to\mathfrak{C}'_B$ with $\text{Eval}_s\circ\mathfrak{F}=\text{Eval}_s$; if $h$ is a homotopy via $\mathfrak{C}_B$ between $f$ and $g$, then $\mathfrak{F}\circ h$ is a homotopy via $\mathfrak{C}'_B$ between $f$ and $g$, cf.\ \cite[Lemma 4.2.36]{FOOO1}.

\begin{prop}\label{prop:comparison}
Let $\mathbb{K}$ be a commutative ring, let $C_1$, $C_2$ be two $\Ain$-algebras, let $\mathfrak{C}_1$ and $\mathfrak{C}_2$ be models for the interval $[0,1]\times C_1$ and $[0,1]\times C_2$, respectively, and let $f\colon C_1\to C_2$ be an $\Ain$-morphism.\ Then there exists an $\Ain$-morphism $\mathfrak{F}\colon\mathfrak{C}_1\to\mathfrak{C}_2$ such that
\begin{align}\label{eq:compatible}
\text{Eval}_s\circ\mathfrak{F}=f\circ\text{Eval}_s\quad (s=0,1),\qquad \mathfrak{F}\circ\text{Incl}=\text{Incl}\circ f .
\end{align}
\end{prop}

We give a proof of Proposition \ref{prop:comparison} in Section \ref{sec:homOp} which is a bit different from the argument of Fukaya-Oh-Ohta-Ono.

\begin{rem}
The argument of \cite[\S 4.4]{FOOO1} works over a ring and does not use any model structure: it is an induction on the number filtration which uses exactly the split injectivity of $\text{Incl}^1$ and the split surjectivity of $(\text{Eval}_0)^1\oplus(\text{Eval}_1)^1$, see \cite[Lemma 4.4.2 and Lemma 4.4.3]{FOOO1}.\ 
Over a field, Theorem \ref{thm:indep} is also an instance of the standard fact that right homotopies can be computed with any good path object, see Proposition \ref{prop:path} and Proposition \ref{prop:indep} below.
\end{rem}

If $B$ is strictly unital and $f\sim g$, then it is well known that $f\approx g$, see for example \cite[Lemma 2.5]{Sei}, \cite[Lemma 2.7]{Orn1}, \cite[Lemma 13.14]{FukeX} or \cite[Corollary 1.23]{COS2}.\ 
The converse is false in general.\ The following example is due to Fukaya, cf.\ \cite[Example 13.15]{FukeX}.

\begin{exmp}\label{ex:fukaya}
Let $C$ be a unital associative $\mathbb{K}$-algebra, regarded as a DG algebra concentrated in degree $0$ (so $m^1_C=0$ and $m^2_C(x_2,x_1)=x_2x_1$).\ We take $A=B=C$ and two unital $\mathbb{K}$-algebra homomorphisms $f,g\colon C\to C$, regarded as strict $\Ain$-morphisms.\ 
First, we note that $f\sim g$ if and only if $f=g$: since $C$ is concentrated in degree $0$, a prenatural transformation $T$ of degree $0$ with $T^0=0$ has $T^n=0$ for every $n$ ($T^n$ has degree $-n$), hence $\mathfrak{M}^1(T)=0$.\ 
On the other hand, $f\approx g$ if and only if there exists an invertible element $c\in C$ such that $f(x)=c^{-1}g(x)c$ for every $x\in C$.\ 
In this case, we take 
\begin{align*}
S:=(S^0=c ,S^{n>0}=0)\colon f&\to g, \quad T:=(T^0=c^{-1} ,T^{n>0}=0)\colon g\to f,\quad W=W'=0.
\end{align*}
Then $S$ and $T$ are natural: the only components of $\mathfrak{M}^1(S)$ and $\mathfrak{M}^1(T)$ which could be non zero are
\begin{align*}
\mathfrak{M}^1(S)^1(x)=m^2_C(g(x),c)-m^2_C(c,f(x))=g(x)c-cf(x)=0,\qquad \mathfrak{M}^1(T)^1(x)=f(x)c^{-1}-c^{-1}g(x)=0,
\end{align*}
and $\mathfrak{M}^2(S,T)=\Id_{g}$, $\mathfrak{M}^2(T,S)=\Id_{f}$.\ 
Taking for instance $g=\Id_C$ and $f(x)=c^{-1}xc$, we get $f\approx g$ and $f\not\sim g$.\\
On the other hand, if $f\approx g$ then by \cite[Lemma 2.8]{Gen} (or \cite[Lemma 2.10]{COS1}) we have a DG functor:
\begin{align*}
\psi\colon C&\to P(C)\\
\bullet&\mapsto \psi^0(\bullet):=(\bullet,\bullet,c)\\
x&\mapsto \psi^1(x):=\Bigl(\:\begin{matrix}
f(x) & 0\\
0 & g(x)
\end{matrix}\:\Bigr)
\end{align*}
such that
\begin{equation*}
\xymatrix{
&C&\\
C\ar[ur]^{f}\ar[dr]_{g}\ar[r]^-{\psi}&{P(C)}\ar[u]_s^{\simeq}\ar[d]^t_{\simeq}&\ar@{=}[ul]\ar@{=}[dl] C\ar[l]_-{p}^-{\simeq}\\
&C&
}
\end{equation*}
Since $s\circ\psi=f$ and $t\circ\psi=g$, $f$ and $g$ are right homotopic through the path object $P(C)$ of the Tabuada model structure (see \cite[\S 3]{Tab2} and Remark \ref{rem:notgood}), but $f\not\sim g$ since $f\neq g$.\ 
There is no contradiction with Theorem \ref{thm:indep}: $P(C)$ is a DG category with many objects, not a model for the interval in the sense of Definition \ref{defn:model}.\ 
The reason for this discrepancy is explained in Remark \ref{rem:notgood}: $I\otimes C$ is a path object for $C$ also in the Tabuada model structure, but not a good one.\ 
On the other hand $P(C)$ is a good path object.
\end{exmp}

\subsection{Elementary homotopic morphisms of 2-colored DG operads}\label{subsec:elemhomot}

An $\Ain$-algebra can also be seen as an algebra over the DG operad $\Ain$, see \cite{LoVa}.\ 
There is a 2-colored version of the operad $\Ain$, denoted by $A^2_\infty$, whose algebras are the $\Ain$-morphisms which is well explained in \cite[\S 4]{LAM}.\ 
We fix two colors: ``blue'' ($b$) and ``red'' ($r$).\ Now we describe the DG operads $A^2_\infty$ and $Hom(V,W)$.\ $A^2_{\infty}$ is defined as follows:
\begin{itemize}
\item[CO1)] Given $n>0$ and $i_1,...,i_n,o\in\{\text{``blue''},\text{``red''}\}$, $A^2_\infty(i_1,...,i_n;o)$ is a graded $\mathbb{K}$-module.\ As a graded 2-colored operad, $A^2_\infty$ is freely generated by the corollas
\begin{equation}\label{eq:corollas}
\corolla{2}{opblue}{opblue},\ \corolla{3}{opblue}{opblue},\ \corolla{4}{opblue}{opblue},\ \dots\qquad
\corolla{2}{opred}{opred},\ \corolla{3}{opred}{opred},\ \corolla{4}{opred}{opred},\ \dots\qquad
\corolla{1}{opblue}{opred}\,,\ \corolla{2}{opblue}{opred},\ \corolla{3}{opblue}{opred},\ \dots
\end{equation}
namely: the blue corollas $\T^{b^{\times n}}_b$ and the red corollas $\T^{r^{\times n}}_r$, with $n\ge 2$ leaves, of shifted degree $1$ (i.e.\ of degree $2-n$, as $m^n$), and the mixed corollas $\T^{b^{\times n}}_r$, with $n\ge 1$ blue leaves and red root, of shifted degree $0$ (i.e.\ of degree $1-n$, as $f^n$):
\begin{align*}
\T^{b^{\times n}}_b=\corollagen{n}{opblue}{opblue},\qquad \T^{r^{\times n}}_r=\corollagen{n}{opred}{opred},\qquad \T^{b^{\times n}}_r=\corollagen{n}{opblue}{opred}.
\end{align*}
The elements of $A^2_\infty(i_1,...,i_n;o)$ are linear combinations of the trees obtained by grafting these corollas, whose leaves are colored by $i_1,...,i_n$ and whose root is colored by $o$.
\item[CO2)] The composition
\begin{align*}
\circ\colon A^2_\infty(o_1,...,o_n;o)\otimes A^2_{\infty}(\underline{i}_1;o_1)\otimes...\otimes A^2_{\infty}(\underline{i}_n;o_n)&\to A^2_{\infty}(\underline{i}_1,...,\underline{i}_n;o)\\
\gamma\otimes \gamma_{1}\otimes...\otimes \gamma_{n}&\mapsto \circ(\gamma;\gamma_{1},...,\gamma_{n}),
\end{align*}
where $\underline{i}_1,...,\underline{i}_n$ are sequences of colors, is given by grafting the root of the tree $\gamma_{k}$ on the $k$-th leaf of the tree $\gamma$, counted from the left.\ 
We write $\gamma\circ_k\gamma'$ for the tree obtained by grafting the root of $\gamma'$ on the $k$-th leaf of $\gamma$, counted from the right (so that $k-1$ leaves of $\gamma$ lie on the right of $\gamma'$, as the inputs $a_{k-1},...,a_1$ in (\ref{eq:Ainf-cat})).
\item[CO3)] The differential 
\begin{align*}
d_{A^2_\infty}\colon A^2_\infty(i_1,...,i_{n};o)\to  A^2_\infty(i_1,...,i_{n};o)
\end{align*}
is the derivation which, on the generators, is given by the equations (\ref{eq:Ainf-cat}) and (\ref{eq:Ainf-fun}) (without the terms containing $m^1$, which give the differential of the Hom complexes), cf.\ 
\cite[\S 4.4.1]{LAM}.\ Namely, for the blue corollas
\begin{align}\label{eq:dblue}
d_{A^2_\infty}\big(\T^{b^{\times n}}_b\big)=-\sum_{j=2}^{n-1}\ \sum_{k=0}^{n-j}\T^{b^{\times(n-j+1)}}_b\circ_{k+1}\T^{b^{\times j}}_b,
\qquad\text{i.e.}\qquad
d_{A^2_\infty}\Big(\corollagen{n}{opblue}{opblue}\Big)=-\sum_{j=2}^{n-1}\ \sum_{k=0}^{n-j}\corollagen{n-j+1}{opblue}{opblue}\circ_{k+1}\corollagen{j}{opblue}{opblue},
\end{align}
the same formula holds for the red corollas $\T^{r^{\times n}}_r=\corollagen{n}{opred}{opred}$, and for the mixed corollas
\begin{equation}\label{eq:dmixed}
\begin{split}
d_{A^2_\infty}\big(\T^{b^{\times n}}_r\big)&=\sum_{j=2}^{n}\ \sum_{k=0}^{n-j}\T^{b^{\times(n-j+1)}}_r\circ_{k+1}\T^{b^{\times j}}_b-\sum_{p=2}^{n}\ \sum_{\substack{s_1+...+s_p=n\\ s_1,...,s_p\ge 1}}\T^{r^{\times p}}_r\circ\big(\T^{b^{\times s_p}}_r,...,\T^{b^{\times s_1}}_r\big),\\
\text{i.e.}\quad d_{A^2_\infty}\Big(\corollagen{n}{opblue}{opred}\Big)&=\sum_{j=2}^{n}\ \sum_{k=0}^{n-j}\corollagen{n-j+1}{opblue}{opred}\circ_{k+1}\corollagen{j}{opblue}{opblue}-\sum_{p=2}^{n}\ \sum_{\substack{s_1+...+s_p=n\\ s_1,...,s_p\ge 1}}\corollagen{p}{opred}{opred}\circ\Big(\corollagen{s_p}{opblue}{opred},...,\corollagen{s_1}{opblue}{opred}\Big).
\end{split}
\end{equation}
Here the $s_i$ can be equal to $1$: the corolla $\T^{b}_r$ corresponds to the linear part $f^1$ of an $\Ain$-morphism; for instance $d_{A^2_\infty}(\T^{b,b}_r)=\T^{b}_r\circ\T^{b,b}_b-\T^{r,r}_r\circ(\T^{b}_r,\T^{b}_r)$, which corresponds to $f^1(m^2_A(a_2,a_1))-m^2_B(f^1(a_2),f^1(a_1))$, i.e.\ to the right hand side of (\ref{eq:starn-d}) for $n=2$.\ In particular $d_{A^2_\infty}$ vanishes on $\T^{b,b}_b$, $\T^{r,r}_r$ and $\T^{b}_r$.
\end{itemize}
In the same vein, given two chain complexes $(V,d_V)$ and $(W,d_W)$, there is a 2-colored DG operad $Hom(V,W)$ defined as follows:
\begin{itemize}
\item[CO1)] Given $n+1$ colors $i_1,...,i_n,o$ we define
\begin{align*}
Hom(V,W)(i_1,...,i_n;o):=\mbox{Hom}_{\mathbb{K}}(\mathbb{V}^{i_1}\otimes...\otimes \mathbb{V}^{i_n},\mathbb{V}^{o}),
\qquad\text{where}\qquad
\mathbb{V}^k=
\begin{cases}
V, & \text{if $k$ is ``blue''},\\
W, & \text{if $k$ is ``red''},
\end{cases}
\end{align*}
is the complex of graded $\mathbb{K}$-linear maps, graded by the shifted degree: a map $\phi\colon\mathbb{V}^{i_1}\otimes...\otimes\mathbb{V}^{i_n}\to\mathbb{V}^{o}$ of degree $\deg(\phi)$ has shifted degree $\deg'(\phi):=\deg(\phi)+n-1$.
\item[CO2)] The composition is the composition of multilinear maps, $\circ(\phi;\phi_{1},...,\phi_{n}):=\phi\circ(\phi_{1}\otimes...\otimes\phi_{n})$ whenever the colors match, with the Koszul sign rule.\ In particular
\begin{align*}
(\phi\circ_k\chi)(x_n,...,x_1)=(-1)^{\deg'(\chi)(\deg'(x_{k-1})+...+\deg'(x_1))}\,\phi\big(x_n,...,x_{k+m},\chi(x_{k+m-1},...,x_{k}),x_{k-1},...,x_1\big),
\end{align*}
where $\chi$ has $m$ inputs.\ Since $m^n$ has shifted degree $1$ and $\Ff^n$ has shifted degree $0$, the signs $(-1)^{\dagger_k}$ in (\ref{eq:Ainf-cat}) and (\ref{eq:Ainf-fun}) are Koszul signs.
\item[CO3)] The differential is the differential of the Hom complex:
\begin{align*}
d_{V,W}(\phi):=d\circ\phi-(-1)^{\deg'(\phi)}\sum_{k=1}^{n}\phi\circ_k d,
\end{align*}
where $d$ denotes $d_V$ or $d_W$, according to the color of the corresponding input or output.\ For $V=A$ and $W=C$ with $d_V=m^1_A$, $d_W=m^1_C$, this is the differential $d_{Hom(A^{\otimes n},C)}$ of Section \ref{sec:background}.
\end{itemize}
If $R$ is a graded commutative DG algebra and $\mathscr{P}$ is a 2-colored DG operad, $\mathscr{P}\otimes R$ is the 2-colored DG operad given by
\begin{align*}
(\phi\otimes r)\circ_k(\chi\otimes s):=(-1)^{\deg(r)\deg'(\chi)}(\phi\circ_k\chi)\otimes rs,\qquad d(\phi\otimes r):=d_{V,W}(\phi)\otimes r+(-1)^{\deg'(\phi)}\phi\otimes d_R(r).
\end{align*}
A DG $A^2_\infty$-algebra is a morphism of 2-colored DG operads
\begin{align}
\rho\colon A^2_\infty\to Hom(V,W).
\end{align}
Roughly speaking, for every color profile $(i_1,...,i_n;o)$ we have chain maps
\begin{align*}
\rho(i_1,...,i_n;o)\colon A^2_\infty(i_1,...,i_n;o)\to Hom(V,W)\big( i_1,...,i_n;o \big)
\end{align*}
preserving the operadic compositions.\ 
More explicitly, a DG $A^2_\infty$-algebra is given by the datum:
\begin{itemize}
\item[1)] two $\Ain$-algebras $(A,m^n_A)$ and $(B,m^n_B)$, with $(A,m^1_A)=(V,d_V)$ and $(B,m^1_B)=(W,d_W)$,
\item[2)] an $\Ain$-morphism $f\colon A\to B$,
\end{itemize}
namely $m^n_A=\rho(b,...,b;b)(\T^{b^{\times n}}_b)$ and $m^n_B=\rho(r,...,r;r)(\T^{r^{\times n}}_r)$ for $n\ge2$, and $f^n=\rho(b,...,b;r)(\T^{b^{\times n}}_r)$ for $n\ge1$.\ 
Given an $\Ain$-morphism $f\colon A\to B$, we denote by $\rho_f\colon A^2_\infty\to Hom(A,B)$ the corresponding morphism of 2-colored DG operads.\ 
A detailed treatment of the 2-colored DG operad $A^2_\infty$ can be found in \cite[\S 4.4.1]{LAM}.\\

Now we can easily adapt the definition of \emph{elementary homotopic} morphisms of DG operads by Markl, Shnider and Stasheff \cite[\S 3.10]{MSS} to morphisms of 2-colored DG operads.

\begin{defn}[Elementary homotopic]
Let $\rho,\rho'\colon \mathscr{O}\to \mathscr{P}$ be morphisms of (2-colored) DG operads.\ We say that $\rho$ and $\rho'$ are \emph{elementary homotopic} if there exists a morphism of (2-colored) DG operads $\tau\colon\mathscr{O}\to\mathscr{P}\otimes R$, with $R=\mathbb{K}[t,dt]$ (defined in Subsection \ref{susec:interval}), fitting the diagram
\begin{align}\label{diag:homoper}
\xymatrix{
&&\mathscr{P}\\
\mathscr{O}\ar[rr]^-{\tau}\ar[rru]^-{\rho}\ar[rrd]_-{\rho'}&&\mathscr{P}\otimes R\ar[u]_{\Id\otimes\text{ev}_0}\ar[d]^{\Id\otimes\text{ev}_1}\\
&&\mathscr{P}
}
\end{align}
cf.\ \cite[Definition 3.121]{MSS}.
\end{defn}

\begin{defn}
Given two $\Ain$-morphisms $f,g\colon A\to B$ we write $f\asymp g$ if $f$ and $g$, considered as morphisms of 2-colored DG operads $A^2_\infty\to Hom(A,B)$, are elementary homotopic.\ 
We say that an elementary homotopy $\tau\colon A^2_\infty\to Hom(A,B)\otimes R$ between $f$ and $g$ is \emph{constant on the monochromatic parts} if
\begin{align}\label{eq:constant}
\tau(b,...,b;b)(\T^{b^{\times n}}_b)=m^n_A\otimes 1_R,\qquad \tau(r,...,r;r)(\T^{r^{\times n}}_r)=m^n_B\otimes 1_R,\qquad n\ge 2,
\end{align}
and we write $f\asymp^{c}g$ if $f$ and $g$ are elementary homotopic via such a $\tau$.\ Clearly $f\asymp^{c}g$ implies $f\asymp g$.
\end{defn}

\section{Path objects and right homotopies of $\Ain$-algebras}\label{sec:path}

In this section, unless otherwise stated, $\mathbb{K}$ is a field and $\mbox{Alg}_{\infty}$ denotes the category of (non unital) $\Ain$-algebras over $\mathbb{K}$ with (non strict) $\Ain$-morphisms.

\begin{rem}\label{rem:cyl}
The category $\mbox{Alg}_{\infty}$ has finite products, but it does not admit colimits in general: neither coequalizers \cite{Orn6} nor coproducts, see \cite[\S 4.1]{Val}.\ 
Therefore the cylinder objects of Quillen's theory, i.e.\ the factorizations of the codiagonal $A\sqcup A\to A$, are not available in $\mbox{Alg}_{\infty}$, while the path objects, i.e.\ the factorizations of the diagonal $\Delta\colon B\to B\times B$, are.\ 
For this reason, for $\Ain$-algebras the natural notion is the \emph{right} homotopy.
\end{rem}

\begin{thm}[{\cite[Th\'eor\`eme 1.3.3.1]{LH}, \cite[Theorem 4.2]{Val}}]\label{thm:LH}
Consider the following classes of morphisms in $\mbox{Alg}_{\infty}$: 
\begin{itemize}
\item the \emph{weak equivalences} are the quasi-isomorphisms, 
\item the \emph{cofibrations} are the $f$ such that $f^1$ is injective 
\item the \emph{fibrations} are the $f$ such that $f^1$ is surjective.\ 
\end{itemize}
These classes satisfy the axioms (MC2)-(MC5) of a model category, $\mbox{Alg}_{\infty}$ has finite products and pullbacks along fibrations, and every $\Ain$-algebra is fibrant and cofibrant.\ 
In particular $\mbox{Alg}_{\infty}$ is a category of fibrant objects in the sense of Brown \cite{Bro}, cf.\ \cite[\S 4]{Val}.
\end{thm}
Note that in \cite{Val} the field is assumed of characteristic zero, since general operads are considered.\ 
On the other hand, for $\Ain$-algebras \cite{LH} works over any field.

\begin{rem}\label{rem:split}
In the construction of the pullback of a fibration $q$ \cite[proof of axiom (A) in Th\'eor\`eme 1.3.3.1]{LH}, \cite[proof of (MC1') in Theorem 4.2 and Proposition 4.4 (2)]{Val}, one uses a section of $q^1$ as a map of graded vector spaces.\ 
Over a commutative ring such a section does not exist in general, and one has to restrict to fibrations whose first component is \emph{split} surjective, cf.\ \cite{CO}.\ 
This is consistent with the way condition MI4) is used in \cite{FOOO1}, see the footnote to Definition \ref{defn:model}.
\end{rem}

\begin{defn}\label{defn:path}
Let $B$ be an $\Ain$-algebra.\ A \emph{path object} for $B$ is a factorization of the diagonal
\begin{align*}
B\xrightarrow{\ r\ }P\xrightarrow{\ q=(q_0,q_1)\ }B\times B,\qquad q\circ r=\Delta=(\Id_B,\Id_B),
\end{align*}
in $\mbox{Alg}_{\infty}$, such that $r$ is a weak equivalence.\ It is \emph{good} if $q$ is a fibration and \emph{very good} if moreover $r$ is a cofibration.\ 
Two $\Ain$-morphisms $f,g\colon A\to B$ are \emph{right homotopic through $P$} if there is an $\Ain$-morphism $F\colon A\to P$ such that $q\circ F=(f,g)$, i.e.\ $q_0\circ F=f$ and $q_1\circ F=g$.
\end{defn}

If $P=\mathfrak{C}_B$ is a model for the interval and $q=\text{Eval}_0\times\text{Eval}_1$, being right homotopic through $P$ is exactly being homotopic via $\mathfrak{C}_B$ in the sense of Definition \ref{larks}.

\begin{prop}\label{prop:path}
Every model for the interval $\mathfrak{C}_B$ (Definition \ref{defn:model}) is a very good path object for $B$ in $\mbox{Alg}_{\infty}$, with $r=\text{Incl}$ and $q=\text{Eval}_0\times\text{Eval}_1$.\ 
In particular $I\otimes B$ is a very good path object for $B$, and so is $B\otimes R$ if $\mathbb{Q}\subset\mathbb{K}$.
\end{prop}

\begin{proof}
By MI2 we have $(\text{Eval}_0\times\text{Eval}_1)\circ\text{Incl}=(\text{Eval}_0\circ\text{Incl},\text{Eval}_1\circ\text{Incl})=(\Id_B,\Id_B)=\Delta$.\ 
By MI3, $\text{Incl}^1$ is a homotopy equivalence, hence a quasi-isomorphism, so $\text{Incl}$ is a weak equivalence.\ 
Since the morphisms are strict (MI1), MI2 gives $(\text{Eval}_0)^1\circ\text{Incl}^1=\Id_B$, hence $\text{Incl}^1$ is injective and $\text{Incl}$ is a cofibration.\ 
Finally the first component of $\text{Eval}_0\times\text{Eval}_1$ is $(\text{Eval}_0)^1\oplus(\text{Eval}_1)^1$, which is surjective by MI4, so $\text{Eval}_0\times\text{Eval}_1$ is a fibration.\ 
The last statement follows from Example \ref{ex:I} and Theorem \ref{thm:R}.
\end{proof}

\begin{prop}\label{prop:indep}
Let $f,g\colon A\to B$ be two $\Ain$-morphisms.\ If $f$ and $g$ are right homotopic through some path object $P$ of $B$, then they are right homotopic through every good path object $P'$ of $B$.\ 
In particular, for any two models $\mathfrak{C}_B$ and $\mathfrak{C}'_B$ for the interval we have $f\sim_{\mathfrak{C}_B}g$ if and only if $f\sim_{\mathfrak{C}'_B}g$.
\end{prop}

\begin{proof}
Let $B\xrightarrow{r}P\xrightarrow{q}B\times B$ be a path object and $F\colon A\to P$ with $q\circ F=(f,g)$, and let $B\xrightarrow{r'}P'\xrightarrow{q'}B\times B$ be a good path object.\ 
We use Theorem \ref{thm:LH}.\ 
By (MC5) we factor $r=p\circ j$, where $j\colon B\to Q$ is a trivial cofibration and $p\colon Q\to P$ is a fibration; by (MC2) $p$ is a weak equivalence, hence a trivial fibration.\ 
Since the zero $\Ain$-algebra is initial and $0\to A$ is a cofibration, (MC4) applied to the square
\begin{align*}
\xymatrix{
0\ar[r]\ar[d]&Q\ar[d]^{p}\\
A\ar[r]^{F}&P
}
\end{align*}
gives $\widetilde{F}\colon A\to Q$ with $p\circ\widetilde{F}=F$.\ 
Since $q'\circ r'=\Delta=q\circ r=(q\circ p)\circ j$, (MC4) applied to the square
\begin{align*}
\xymatrix{
B\ar[r]^{r'}\ar[d]_{j}&P'\ar[d]^{q'}\\
Q\ar[r]^{q\circ p}&B\times B
}
\end{align*}
(where $j$ is a trivial cofibration and $q'$ is a fibration) gives $\lambda\colon Q\to P'$ with $\lambda\circ j=r'$ and $q'\circ\lambda=q\circ p$.\ 
Then $F':=\lambda\circ\widetilde{F}\colon A\to P'$ satisfies $q'\circ F'=q\circ p\circ\widetilde{F}=q\circ F=(f,g)$.\ 
The last statement follows from Proposition \ref{prop:path}, since models for the interval are good path objects.
\end{proof}

\begin{rem}\label{rem:FOOO}
Proposition \ref{prop:indep} recovers Theorem \ref{thm:indep} over a field, i.e.\ it gives a short proof of \cite[Lemma 4.2.36]{FOOO1} in the particular case of (unfiltered) $\Ain$-algebras over a field.\ 
The statements of \cite[Theorem 4.2.34]{FOOO1} concerning filtered, gapped, unital or strict morphisms are not covered by this argument, since we do not have a model structure (or a similar structure) on the corresponding categories.\ 
On the other hand, the notion of path object is much more general than the one of model for the interval $[0,1]\times B$, 
since path objects can be taken with non strict $\Ain$-morphisms $r$ and $q$.\ 
\end{rem}

\begin{cor}\label{cor:sim-right}
Let $\mathbb{K}$ be a commutative ring and $f,g\colon A\to B$ two (non unital) $\Ain$-morphisms.\ The following are equivalent:
\begin{itemize}
\item[(1)] $f\sim g$;
\item[(2)] $f\sim_{I\otimes B}g$.
\end{itemize}
If $\mathbb{Q}\subset\mathbb{K}$, they are also equivalent to
\begin{itemize}
\item[(3)] $f\sim_{B\otimes R}g$;
\end{itemize}
and if $\mathbb{K}$ is a field, they are also equivalent to
\begin{itemize}
\item[(4)] $f$ and $g$ are right homotopic through some (equivalently, any good) path object of $B$ in $\mbox{Alg}_{\infty}$.
\end{itemize}
\end{cor}

\begin{proof}
(1)$\Leftrightarrow$(2) is the characterization of $\sim$ via $I\otimes B$ recalled in (\ref{diag:pathint}).\ 
$I\otimes B$ is a model for the interval $[0,1]\times B$ (Example \ref{ex:I}) and, if $\mathbb{Q}\subset\mathbb{K}$, so is $B\otimes R$ (Theorem \ref{thm:R}); hence (2)$\Leftrightarrow$(3) by Theorem \ref{thm:indep}.\ 
If $\mathbb{K}$ is a field, (2)$\Leftrightarrow$(4) follows from Propositions \ref{prop:path} and \ref{prop:indep}.
\end{proof}

\begin{rem}[Why $f\approx g\not\Rightarrow f\sim g$]\label{rem:notgood}
Let $B$ be a DG algebra, seen as a DG category with one object, and consider the Tabuada model structure on the category of small DG categories over a commutative ring $\mathbb{K}$ \cite{Tab1}.\ We recall that a DG functor $F\colon\mathcal{C}\to\mathcal{D}$ is a fibration if and only if
\begin{itemize}
\item[(F1)] $F$ is surjective on the complexes of morphisms,
\item[(F2)] for every object $c\in\mathcal{C}$ and every $v\in\mathcal{D}(F(c),d)$ which becomes invertible in $H^0(\mathcal{D})$, there is $u\in\mathcal{C}(c,c')$ which becomes invertible in $H^0(\mathcal{C})$ and such that $F(u)=v$,
\end{itemize}
see \cite[Proposition 1.13]{TabTh}.\ 
The DG functors
\begin{align*}
B\xrightarrow{\ \text{Incl}\ }I\otimes B\xrightarrow{\ \text{Eval}_0\times\text{Eval}_1\ }B\times B
\end{align*}
form a path object for $B$ also in this model structure: $\text{Incl}$ is bijective on objects and a homotopy equivalence on morphisms (Example \ref{ex:I}), hence a quasi-equivalence, and the composition is the diagonal.\ 
However $\text{Eval}_0\times\text{Eval}_1$ satisfies (F1) but \emph{not} (F2) if $H^0(B)$ contains two different invertible elements $[u]\neq[w]$.\ 
So $I\otimes B$ is a path object which is \emph{not good} in the Tabuada model structure, while $P(B)\to B\times B$ is a fibration, i.e.\ $P(B)$ is a good path object \cite[\S 3]{Tab2}.\\

For a cofibrant source, the homotopy relation of a model category can be computed with any fixed good path object, but not with an arbitrary path object.\ 
Now, right homotopies (by $\Ain$-functors) through $I\otimes B$ give exactly the relation $\sim$, while $f\approx g$ gives a right homotopy (by an $\Ain$-functor) through $P(B)$, see Remark \ref{diag:approx}.\ 
Example \ref{ex:fukaya} shows that the first relation is strictly finer.\ There the obstruction is visible directly: a right homotopy through $I\otimes C$ between the ring homomorphisms $f,g\colon C\to C$ is an $\Ain$-morphism $H\colon C\to I\otimes C$ with $H^1(x)=u_0\otimes f(x)+u_1\otimes g(x)$ and $H^{n}=0$ for $n\ge 2$ (for degree reasons, since $C$ is concentrated in degree $0$), and $d(H^1(x))=h\otimes(f(x)-g(x))$ vanishes only if $f=g$; the invertible element $c$ with $f=c^{-1}gc$, which is part of the object $\psi^0(\bullet)$ of $P(C)$, has no counterpart in $I\otimes C$.\ 

The argument of the proof of Proposition \ref{prop:indep}, run in the Tabuada model structure (for a cofibrant source), transports a right homotopy through $I\otimes B$ to a right homotopy through the good path object $P(B)$, in accordance with $\sim\Rightarrow\approx$; but it cannot be run in the opposite direction.\ 
For the same reason, the fibrancy of the relative category $(\Ain\mbox{Cat},W^{\tiny\Ain}_{\tiny\mbox{qe}})$ in $\mbox{RelCat}$ \cite{CO} does not imply that all the path objects of $B$ in the weak sense define the same homotopy relation: only those whose second map is a fibration do.\ 
Note that $I\otimes B$ is not even a path object in the sense of Brown \cite{Bro}, where the second map is required to be a fibration.

On the other hand, in $\mbox{Alg}_{\infty}$ the weak equivalences are the quasi-isomorphisms and there are no objects to move: condition (F2) does not appear and $I\otimes B$ is a good path object (Proposition \ref{prop:path}).\ 
This is why $\sim$ is the homotopy relation of $\mbox{Alg}_{\infty}$ (Corollary \ref{cor:sim-right}), whereas $\approx$ is the relevant one for $\Ain$ (or DG) categories up to quasi-equivalence, cf.\ \cite[Theorem 6.18]{Orn2}.
\end{rem}

\section{Homotopic morphisms of 2-colored DG operads}\label{sec:homOp}

In this section we compare the homotopy via the models $I\otimes B$ and $B\otimes R$ with the elementary homotopy constant on the monochromatic parts (Theorem \ref{thm:MSS}).\ 
It follows that, if $\mathbb{Q}\subset\mathbb{K}$, two $\Ain$-morphisms are homotopic if and only if, considered as morphisms of 2-colored DG operads, they are elementary homotopic through a homotopy constant on the monochromatic parts (Corollary \ref{cor:equiv}).\ 
Finally, we give an example of two $\Ain$-morphisms which are elementary homotopic but not homotopic (Example \ref{ex:MSS}).

\begin{thm}\label{thm:MSS}
Let $\mathbb{K}$ be a commutative ring, let $A$, $B$ be two $\Ain$-algebras and let $f,g\colon A\to B$ be two $\Ain$-morphisms.
\begin{itemize}
\item[(1)] If $f\asymp^{c}g$, then there is an $\Ain$-morphism $F\colon A\to B\otimes R$ with $\text{Eval}_0\circ F=f$ and $\text{Eval}_1\circ F=g$; in particular $f\sim_{B\otimes R}g$ if $\mathbb{Q}\subset\mathbb{K}$.
\item[(2)] If $f\sim_{I\otimes B}g$, then $f\asymp^{c}g$.
\end{itemize}
In particular, if $\mathbb{Q}\subset\mathbb{K}$, then $f\sim_{I\otimes B}g\iff f\asymp^{c}g\iff f\sim_{B\otimes R}g$.
\end{thm}

The last statement follows from (1), (2) and Theorem \ref{thm:indep}, since $I\otimes B$ and $B\otimes R$ are models for the interval $[0,1]\times B$ when $\mathbb{Q}\subset\mathbb{K}$.\ 
By Corollary \ref{cor:sim-right} we obtain:

\begin{cor}\label{cor:equiv}
Suppose that $\mathbb{Q}\subset\mathbb{K}$ and let $A$, $B$ be two (non unital) $\Ain$-algebras.\ Then two $\Ain$-morphisms $f,g\colon A\to B$ satisfy $f\sim g$ if and only if $f\asymp^{c}g$.
\end{cor}

To prove Theorem \ref{thm:MSS} we need some preliminaries.\ 
Let $V$ be a complex (in our applications $V=A$ or $V=I\otimes B$, with the differential $m^1$) and let $n\ge1$.\ For a homogeneous $x=(v_n,...,v_1)\in V^{\otimes n}$ we put $\deg'(x):=\deg'(v_n)+...+\deg'(v_1)$, and we consider the map
\begin{align*}
\Phi\colon Hom(V^{\otimes n},B)\otimes R\to Hom(V^{\otimes n},B\otimes R),\qquad \Phi(\varphi\otimes r)(v_n,...,v_1):=(-1)^{\deg(r)\deg'(x)}\,\varphi(v_n,...,v_1)\otimes r .
\end{align*}
Here $Hom(V^{\otimes n},B)$ and $Hom(V^{\otimes n},B\otimes R)$ are the Hom complexes of Section \ref{sec:background}.\
Note that $\Phi$ is injective but, in general, it is not surjective.\ 
We say that a map $G\colon V^{\otimes n}\to B\otimes R$ is \emph{of finite type} if it is in the image of $\Phi$.\ 
The maps of finite type form a subcomplex, and we have:
\begin{itemize}
\item[(a)] if $G\colon V^{\otimes n}\to B\otimes R$ is of finite type, then so is $G\circ\alpha$ for every $\mathbb{K}$-linear map $\alpha\colon W^{\otimes m}\to V^{\otimes n}$.
\item[(b)] if $G_1,...,G_p$ are of finite type, then so is $m^p_{B\otimes R}\circ(G_p\otimes...\otimes G_1)$.
\item[(c)] Now $V=A$, $\Phi$ is compatible with the compositions with $m_A$ and $m_B$: we have
\begin{align}\label{eq:c1}
\Phi\big((\varphi\otimes r)\circ_k(\chi\otimes 1_R)\big)=\Phi(\varphi\otimes r)\circ_k\chi
\end{align}
for $\chi\in Hom(A^{\otimes j},A)$, $\varphi\otimes r\in Hom(A^{\otimes l},B)\otimes R$ and $1\le k\le l$, and, for $p\ge2$ and $\gamma_1,...,\gamma_p$ in $Hom(A^{\otimes s_1},B)\otimes R,...,Hom(A^{\otimes s_p},B)\otimes R$,
\begin{align}\label{eq:c2}
\Phi\big((m^p_B\otimes 1_R)\circ(\gamma_p,...,\gamma_1)\big)=m^p_{B\otimes R}\circ\big(\Phi(\gamma_p),...,\Phi(\gamma_1)\big),
\end{align}
where the compositions on the left are those of the 2-colored operad $Hom(A,B)\otimes R$ of Subsection \ref{subsec:elemhomot}, and those on the right are compositions of multilinear maps.
\end{itemize}

In the next lemma we consider $I\otimes B$ and $B\otimes R$ with the maps $\text{Eval}_s$ and $\text{Incl}$ of Example \ref{ex:I} and Theorem \ref{thm:R}.\ 
Note that $B\otimes R$ satisfies MI1, MI2 and MI4 over any commutative ring $\mathbb{K}$, and that in the proof of Proposition \ref{prop:comparison} below the axiom MI3 is used only for the source.

\begin{lemma}\label{lem:finite}
Let $\mathbb{K}$ be a commutative ring.\ There exists an $\Ain$-morphism $\mathfrak{F}\colon I\otimes B\to B\otimes R$ such that $\text{Eval}_s\circ\mathfrak{F}=\text{Eval}_s$ ($s=0,1$), $\mathfrak{F}\circ\text{Incl}=\text{Incl}$, and $\mathfrak{F}^k$ is of finite type for every $k\ge 1$.
\end{lemma}

Lemma \ref{lem:finite} follows from (a refinement of) the proof of Proposition \ref{prop:comparison}, which we now give.\ 
Recall that, given an $\Ain$-morphism $f\colon C_1\to C_2$ and two models $\mathfrak{C}_1$, $\mathfrak{C}_2$ for the interval $[0,1]\times C_1$ and $[0,1]\times C_2$, respectively, we want an $\Ain$-morphism $\mathfrak{F}\colon\mathfrak{C}_1\to\mathfrak{C}_2$ making the diagram
\begin{align}\label{eq:condizzielle}
\xymatrix{
C_1\ar[d]^{\text{Incl}}\ar[r]^f& C_2\ar[d]^{\text{Incl}}\\
\mathfrak{C}_1 \ar[d]^{\text{Eval}_{s}}\ar@{-->}[r]^{\mathfrak{F}}& \mathfrak{C}_2\ar[d]^{\text{Eval}_{s}}\\
C_1\ar[r]^f& C_2
}
\end{align}
commutative for $s=0,1$, i.e.\ satisfying (\ref{eq:compatible}).

\begin{proof}[Proof of Proposition \ref{prop:comparison}]
We write $E:=(\text{Eval}_0)^1\oplus(\text{Eval}_1)^1\colon\mathfrak{C}_2\to C_2\oplus C_2$ and $J_2:=\ker(E)$, which is a subcomplex of $\mathfrak{C}_2$.\\ 
For $n\ge 1$ we use the following notation: $D_n:=\mathfrak{C}_1^{\otimes n}$, $i_n:=(\text{Incl}^1)^{\otimes n}\colon C_1^{\otimes n}\to D_n$, $r_n:=\big((\text{Eval}_0)^1\big)^{\otimes n}\colon D_n\to C_1^{\otimes n}$, and $Q_n:=\text{coker}(i_n)$.\ 
By MI2 we have $r_n\circ i_n=\Id$, so $i_n$ is split injective, and by MI3 $i_n$ is a homotopy equivalence.\ 
Since the sequence $0\to C_1^{\otimes n}\to D_n\to Q_n\to 0$ is degreewise split, $Q_n$ is homotopy equivalent to the cone of $i_n$, hence it is contractible.\ 
Therefore the complex
\begin{align*}
M^0_n:=\big\{G\in Hom(D_n,\mathfrak{C}_2)\ \big|\ E\circ G=0,\ G\circ i_n=0\big\}=Hom(Q_n,J_2)
\end{align*}
is contractible.\ 
We construct $\mathfrak{F}^n$ by induction on $n\ge 1$.\\ 
Suppose that $\mathfrak{F}^1,...,\mathfrak{F}^{n-1}$ satisfy $(\star_m)$ and (\ref{eq:compatibleK}) for $m<n$.\ 
Since $\text{Eval}_s$ and $\text{Incl}$ are strict, the conditions (\ref{eq:compatible}) for $n$ are equivalent to
\begin{align}\label{eq:compatibleK}
E\circ\mathfrak{F}^n=\big(f^n\circ((\text{Eval}_0)^1)^{\otimes n},f^n\circ((\text{Eval}_1)^1)^{\otimes n}\big),\qquad \mathfrak{F}^n\circ i_n=\text{Incl}^1\circ f^n .
\end{align}
Let $\sigma$ be a section of $E$ (MI4) and put $G_0:=\sigma\circ\big(f^n\circ((\text{Eval}_0)^1)^{\otimes n},f^n\circ((\text{Eval}_1)^1)^{\otimes n}\big)$ and
\begin{align*}
G:=G_0-\big(G_0\circ i_n-\text{Incl}^1\circ f^n\big)\circ r_n .
\end{align*}
By MI2, $E\circ\text{Incl}^1=(\Id,\Id)$ and $E\circ G_0\circ i_n=(f^n,f^n)$, hence $G_0\circ i_n-\text{Incl}^1\circ f^n$ takes values in $J_2$.\ It follows that $G$ satisfies (\ref{eq:compatibleK}).\\
Let $\text{Err}_n:=\text{Err}_n(\mathfrak{F})$ (with $\text{Err}_1:=0$); it is a cycle.\ 
Since $\text{Eval}_s$ is a strict $\Ain$-morphism, $(\text{Eval}_s)^1\circ\text{Err}_n(\mathfrak{F})=\text{Err}_n(\text{Eval}_s\circ\mathfrak{F})$, and $\text{Eval}_s\circ\mathfrak{F}$ coincides for $m<n$ with the $\Ain$-morphism $f\circ\text{Eval}_s$.\\ 
Hence, by $(\star_n)$ for $f\circ\text{Eval}_s$ and by (\ref{eq:compatibleK}) for $G$,
\begin{align*}
(\text{Eval}_s)^1\circ\text{Err}_n=d\big(f^n\circ((\text{Eval}_s)^1)^{\otimes n}\big)=(\text{Eval}_s)^1\circ d(G).
\end{align*}
In the same way, since $\text{Incl}$ is strict, $\text{Err}_n(\mathfrak{F})\circ i_n=\text{Err}_n(\mathfrak{F}\circ\text{Incl})=\text{Err}_n(\text{Incl}\circ f)=d\big(\text{Incl}^1\circ f^n\big)=d(G)\circ i_n$.\ 
Thus $\text{Err}_n-d(G)$ is a cycle of the contractible complex $M^0_n$, and there is $X\in M^0_n$ with $d(X)=\text{Err}_n-d(G)$.\ 
We put $\mathfrak{F}^n:=G+X$.\ Then $d(\mathfrak{F}^n)=\text{Err}_n$, i.e.\ $(\star_n)$ holds, and $\mathfrak{F}^n$ satisfies (\ref{eq:compatibleK}), since $G$ does and $X\in M^0_n$.
\end{proof}

\begin{proof}[Proof of Lemma \ref{lem:finite}]
It suffices to run the proof of Proposition \ref{prop:comparison}, setting $\mathfrak{C}_1=I\otimes B$, $\mathfrak{C}_2=B\otimes R$ and $f=\Id_B$, and check that all the choices can be made of finite type.\\ 
As before, we denote $D_n=(I\otimes B)^{\otimes n}$, $i_n=(i\otimes\Id_B)^{\otimes n}\colon B^{\otimes n}\to D_n$, $r_n=(p^1_0\otimes\Id_B)^{\otimes n}$, and $E=(\text{Eval}_0)^1\oplus(\text{Eval}_1)^1=\Id_B\otimes(\text{ev}_0,\text{ev}_1)\colon B\otimes R\to B\oplus B$; note that $E\circ\Phi(\psi\otimes r)=\big(\text{ev}_0(r)\psi,\text{ev}_1(r)\psi\big)$.\\ 
Suppose that $\mathfrak{F}^1,...,\mathfrak{F}^{n-1}$ are of finite type.\
First, we take the section $\sigma(b_0,b_1):=b_0\otimes(1-t)+b_1\otimes t$ of $E$.\ Then
\begin{align*}
G_0=\Phi\big(f^n\circ(p^1_0\otimes\Id_B)^{\otimes n}\otimes(1-t)+f^n\circ(p^1_1\otimes\Id_B)^{\otimes n}\otimes t\big)\quad\text{and}\quad \text{Incl}^1\circ f^n=\Phi(f^n\otimes 1)
\end{align*}
are of finite type, hence so is $G=G_0-(G_0\circ i_n-\text{Incl}^1\circ f^n)\circ r_n$, by (a).\ 
By (a) and (b), $\text{Err}_n$ is of finite type, and so is $d(G)$.\ Hence $\text{Err}_n-d(G)$ is a cycle of the subcomplex
\begin{align*}
M^{0,\text{ft}}_n:=M^0_n\cap Hom^{\text{ft}}(D_n,B\otimes R)\subset M^0_n .
\end{align*}
Let $J:=\ker(\text{ev}_0,\text{ev}_1)\subset R$, a subcomplex of $R$.\ As graded modules $R=J\oplus\mathbb{K}\,1\oplus\mathbb{K}\,t$.\ 
It is not hard to prove that $\Phi$ restricts to an isomorphism of complexes
\begin{align*}
Hom(Q_n,B)\otimes J\ \xrightarrow{\ \cong\ }\ M^{0,\text{ft}}_n.
\end{align*}
By the proof of Proposition \ref{prop:comparison}, $Q_n$ is contractible.\ Hence $Hom(Q_n,B)$ is contractible, and so is $Hom(Q_n,B)\otimes J$.\ By (iii), $M^{0,\text{ft}}_n$ is contractible.\
Therefore there is $X\in M^{0,\text{ft}}_n$ with $d(X)=\text{Err}_n-d(G)$, and $\mathfrak{F}^n=G+X$ is of finite type.
\end{proof}

In order to lighten the notation from now on, we write $\tau(\T)$ instead of $\tau(c_1,...,c_n;c)(\T)$.

\begin{proof}[Proof of Theorem \ref{thm:MSS}]
We assume that $f\asymp^c g$ and we call  
$$\tau\colon A^2_\infty\to Hom(A,B)\otimes R$$ 
be an elementary homotopy from $\rho_f$ to $\rho_g$ which is constant on the monochromatic parts, i.e.\ $(\Id\otimes\text{ev}_0)\circ\tau=\rho_f$, $(\Id\otimes\text{ev}_1)\circ\tau=\rho_g$ and (\ref{eq:constant}) holds.\ For $n\ge1$ we put
\begin{align*}
F^n:=\Phi\big(\tau(\T^{b^{\times n}}_r)\big)\in Hom(A^{\otimes n},B\otimes R).
\end{align*}
Since $\tau(\T^{b^{\times n}}_r)$ has shifted degree $0$, $F^n$ has degree $1-n$.\ Since $\Phi$ is a morphism of complexes and $\tau$ commutes with the differentials,
\begin{align*}
d(F^n)=\Phi\big(d\,\tau(\T^{b^{\times n}}_r)\big)=\Phi\big(\tau(d_{A^2_\infty}\T^{b^{\times n}}_r)\big).
\end{align*}
By (\ref{eq:dmixed}) and (\ref{eq:constant}),
\begin{align*}
\tau\big(d_{A^2_\infty}\T^{b^{\times n}}_r\big)=\sum_{j=2}^{n}\ \sum_{k=0}^{n-j}\tau\big(\T^{b^{\times(n-j+1)}}_r\big)\circ_{k+1}\big(m^j_A\otimes1_R\big)-\sum_{p=2}^{n}\ \sum_{\substack{s_1+...+s_p=n\\ s_1,...,s_p\ge1}}\big(m^p_B\otimes1_R\big)\circ\big(\tau(\T^{b^{\times s_p}}_r),...,\tau(\T^{b^{\times s_1}}_r)\big),
\end{align*}
hence, by (\ref{eq:c1}) and (\ref{eq:c2}),
\begin{align*}
\Phi\big(\tau(d_{A^2_\infty}\T^{b^{\times n}}_r)\big)=\sum_{j=2}^{n}\ \sum_{k=0}^{n-j}F^{n-j+1}\circ_{k+1}m^j_A-\sum_{p=2}^{n}\ \sum_{\substack{s_1+...+s_p=n\\ s_1,...,s_p\ge1}}m^p_{B\otimes R}\circ\big(F^{s_p},...,F^{s_1}\big).
\end{align*}
Evaluated on $(a_n,...,a_1)$, the right hand side is exactly $\text{Err}_n(F)$ of (\ref{eq:starn-d}), with $C=B\otimes R$. 
Therefore $d(F^n)=\text{Err}_n(F)$ for every $n\ge1$, i.e.\ $F=(F^n)_{n\ge1}$ is an $\Ain$-morphism $A\to B\otimes R$.\ 
Finally, $\text{Eval}_s$ is strict and $(\text{Eval}_s)^1\circ\Phi(\varphi\otimes r)=\text{ev}_s(r)\varphi$, since $\text{ev}_s$ vanishes on $\mathbb{K}[t]\,dt$; hence
\begin{align*}
(\text{Eval}_s\circ F)^n=(\text{Eval}_s)^1\circ F^n=\big((\Id\otimes\text{ev}_s)\circ\tau\big)(\T^{b^{\times n}}_r),
\end{align*}
which is $f^n$ for $s=0$ and $g^n$ for $s=1$.\ So $\text{Eval}_0\circ F=f$ and $\text{Eval}_1\circ F=g$.\ If $\mathbb{Q}\subset\mathbb{K}$, $B\otimes R$ is a model for the interval (Theorem \ref{thm:R}), $f\sim_{B\otimes R}g$ and we are done.\\

Now we assume that $f\sim_{I\otimes B}g$: It implies that there is an $\Ain$-morphism $\mathscr{H}\colon A\to I\otimes B$ with $\text{Eval}_0\circ\mathscr{H}=f$ and $\text{Eval}_1\circ\mathscr{H}=g$.\ 
Let $\mathfrak{F}$ be as in Lemma \ref{lem:finite} and $F:=\mathfrak{F}\circ\mathscr{H}\colon A\to B\otimes R$.\ Then $\text{Eval}_s\circ F=\text{Eval}_s\circ\mathscr{H}$, which is $f$ for $s=0$ and $g$ for $s=1$, and every component
\begin{align*}
F^n=\sum_{k\ge1}\ \sum_{i_1+...+i_k=n}\mathfrak{F}^k\circ\big(\mathscr{H}^{i_k}\otimes...\otimes\mathscr{H}^{i_1}\big)
\end{align*}
Since $A^2_\infty$ is free as a graded 2-colored operad, we can define a morphism of graded 2-colored operads $\tau\colon A^2_\infty\to Hom(A,B)\otimes R$ on the generators by
\begin{align*}
\tau(\T^{b^{\times n}}_b):=m^n_A\otimes 1_R,\qquad \tau(\T^{r^{\times n}}_r):=m^n_B\otimes 1_R\quad (n\ge2),\qquad \tau(\T^{b^{\times n}}_r):=\Phi^{-1}(F^n)\quad(n\ge1).
\end{align*}
It satisfies (\ref{eq:constant}), and we check that it commutes with the differentials on the generators.\ 
On the monochromatic parts, $\tau$ is the composition of $\rho_A$ (resp.\ $\rho_B$) with $x\mapsto x\otimes 1_R$, which is a morphism of DG operads since $\mathbb{K}\to R$ is a morphism of DG algebras.\ 
On the mixed generators, the previous computation, gives
\begin{align*}
\Phi\big(\tau(d_{A^2_\infty}\T^{b^{\times n}}_r)\big)=\text{Err}_n(F),\qquad \Phi\big(d\,\tau(\T^{b^{\times n}}_r)\big)=d(F^n),
\end{align*}
and $d(F^n)=\text{Err}_n(F)$ since $F$ is an $\Ain$-morphism; since $\Phi$ is injective, $d\,\tau(\T^{b^{\times n}}_r)=\tau(d_{A^2_\infty}\T^{b^{\times n}}_r)$.\ 
Finally, $\big((\Id\otimes\text{ev}_s)\circ\tau\big)(\T^{b^{\times n}}_r)=(\text{Eval}_s)^1\circ F^n=(\text{Eval}_s\circ F)^n$, so $(\Id\otimes\text{ev}_0)\circ\tau=\rho_f$ and $(\Id\otimes\text{ev}_1)\circ\tau=\rho_g$.
\end{proof}

\begin{rem}
Note that we do not have a morphism of 2-colored DG operads 
$$Hom(A,B)\otimes R\to Hom(A,B\otimes R)$$
extending $\Phi$.\ On the blue part there is no natural morphism $Hom(A^{\otimes n},A)\otimes R\to Hom(A^{\otimes n},A)$ compatible with the compositions.\ This is why we only use the compatibilities (\ref{eq:c1}) and (\ref{eq:c2}).
\end{rem}

The constancy condition (\ref{eq:constant}) cannot be dropped: the conditions $(\Id\otimes\text{ev}_0)\circ\tau=\rho_f$ and $(\Id\otimes\text{ev}_1)\circ\tau=\rho_g$ only say that $\tau(\T^{b^{\times n}}_b)$ and $\tau(\T^{r^{\times n}}_r)$ are equal to $m^n_A$ and $m^n_B$ for $t=0$ and $t=1$; in particular they do not see the components in $dt$.\ 

\begin{exmp}[$f\asymp g\not\Rightarrow f\approx g$]\label{ex:MSS}
Let $\mathbb{K}$ be a commutative ring.
\begin{itemize}
\item[$\bullet$] Let $A:=\mathbb{K}\,a$, with $\deg(a)=0$, endowed with the trivial $\Ain$-structure, i.e. $m^n_A=0$ for every $n\ge 1$.
\item[$\bullet$] Let $V:=\mathbb{K}\,e\oplus\mathbb{K}\,u$, with $\deg(e)=0$ and $\deg(u)=-1$, and $B:=\mathbb{K}\,1\oplus V$ such that $1$ is the unit\footnote{Note that since $B$ is strictly unital this is an Example of $f\asymp g\not\Rightarrow f\approx g$ and so $f\asymp g\not\Rightarrow f\sim g$.\ On the other hand, $B:=V_+$, i.e. the augmentation of the $\Ain$-algebra $V$.\ One can provide a simplified version of Example \ref{ex:MSS} taking $V$ instead of $B$ to prove $f\asymp g\not\Rightarrow f\sim g$.} and $V\cdot V=0$, with zero differential, regarded as a strictly unital $\Ain$-algebra as in Definition \ref{defn:DG}: $m^1_B=0$, $m^n_B=0$ for $n\ge3$ and $m^2_B(b_2,b_1)=(-1)^{\deg(b_1)}b_2\cdot b_1$, i.e.
\begin{align*}
m^2_B(x,1)=x,\qquad m^2_B(1,y)=(-1)^{\deg(y)}y,\qquad m^2_B(x,y)=0\ \text{ for } x,y\in V .
\end{align*}
\item[$\bullet$] Let $f:=(f^1,0,0,...)$ and $g:=(f^1,g^2,0,0,...)$, where 
\begin{align*}
f^1(a):=e,\qquad g^2(a, a):=-u.
\end{align*}
\end{itemize}
Then $f$ and $g$ are $\Ain$-morphisms, $f\not\approx g$, and $f\asymp g$.
\end{exmp}
\begin{proof}
First, we prove that $f$ and $g$ are $\Ain$-morphisms:\ Since $m_A=0$, $m^1_B=0$ and $m^n_B=0$ for $n\ge3$, the condition $(\star_n)$ reduces to $\sum_{s_1+s_2=n}m^2_B\big(F^{s_2}(...),F^{s_1}(...)\big)=0$; for $F=f$ and $F=g$ all the values of the $F^s$ lie in $V$, and $m^2_B$ vanishes on $V\otimes V$.\\
\emph{$f\not\approx g$.} Let $H\colon f\to g$ be a natural transformation of degree $0$.\ For degree reasons $H^0=\lambda1+\mu e$ for some $\lambda,\mu\in\mathbb{K}$, $H^1(a)\in B^{-1}=\mathbb{K}\,u$ and $H^n=0$ for $n\ge2$.\ Since $m_A=0$, $m^1_B=0$, $m^n_B=0$ for $n\ge3$ and $f^2=0$, the formula for $\mathfrak{M}^1$ gives
\begin{align*}
\mathfrak{M}^1(H)^2(a,a)&=m^2_B\big(g^1(a),H^1(a)\big)-m^2_B\big(H^1(a),f^1(a)\big)+m^2_B\big(g^2(a,a),H^0\big)+m^2_B\big(H^0,f^2(a,a)\big)\\
&=0-0+m^2_B(-u,\lambda1+\mu e)+0=-\lambda u,
\end{align*}
since $m^2_B$ vanishes on $V\otimes V$.\ Hence $\mathfrak{M}^1(H)=0$ implies $\lambda=0$, i.e.\ $H^0=\mu e$.\ 
Now suppose that $f\approx g$, and let $H$, $H'$, $W$, $W'$ be as in (\ref{appross}).\ The component of degree $0$ of the first equation in (\ref{appross}) gives $\pm\,m^2_B(H'^0,H^0)=1+m^1_B(W^0)=1$, which is impossible, since $H^0=\mu e$ and $m^2_B(B^0,e)\subseteq\mathbb{K}\,e$.\ Hence $f\not\approx g$.\ 
Since $B$ is strictly unital, $f\sim g$ would imply $f\approx g$, so $f\not\sim g$.\\
\emph{$f\asymp g$.} 
Let $\psi\colon B\otimes B\to B$ be the map of degree $-1$ (i.e.\ of shifted degree $0$) defined on the basis by
\begin{align*}
\psi(e,e):=u,\qquad \psi(x,y):=0\ \text{ for all the other basis elements } x,y\in\{1,e,u\}.
\end{align*}
We first check by hand that, for all $x_1,x_2,x_3\in B$,
\begin{equation}\label{eq:C}
\begin{split}
C(x_3,x_2,x_1):=\ &m^2_B\big(x_3,\psi(x_2,x_1)\big)+m^2_B\big(\psi(x_3,x_2),x_1\big)\\
&-\psi\big(x_3,m^2_B(x_2,x_1)\big)-(-1)^{\deg'(x_1)}\psi\big(m^2_B(x_3,x_2),x_1\big)=0 .
\end{split}
\end{equation}
It is enough to check it on the basis elements.\ Since $\psi$ vanishes as soon as one of its arguments is $1$ or $u$, and $m^2_B$ vanishes on $V\otimes V$, every term of $C$ vanishes on the triples of basis elements, except on $(1,e,e)$, $(e,1,e)$ and $(e,e,1)$, where all the arguments have degree $0$, i.e.\ $\deg'(x_1)=-1$.\ There we get
\begin{align*}
C(1,e,e)&=m^2_B(1,u)+0-0+\psi(e,e)=-u+u=0,\\
C(e,1,e)&=0+0-\psi(e,e)+\psi(e,e)=0,\\
C(e,e,1)&=0+m^2_B(u,1)-\psi(e,e)+0=u-u=0.
\end{align*}
Let $\varphi:=\psi\circ(f^1\otimes f^1)\colon A\otimes A\to B$, i.e.\ $\varphi(a,a)=u$; it has shifted degree $0$, and $g^2=-\varphi$.\\
Since $A^2_{\infty}$ is free as a graded 2-colored operad on the corollas $\T^{b^{\times n}}_b$, $\T^{r^{\times n}}_r$ ($n\ge 2$) and $\T^{b^{\times n}}_r$ ($n\ge 1$), we can define a morphism of graded 2-colored operads $\tau\colon A^2_{\infty}\to Hom(A,B)\otimes R$ by
\begin{align*}
&\tau(\T^{b^{\times n}}_b):=m^n_A\otimes 1=0 \quad (n\ge 2),\\
&\tau(\T^{r,r}_r):=m^2_B\otimes 1+\psi\otimes dt,\qquad \tau(\T^{r^{\times n}}_r):=0\quad (n\ge 3),\\
&\tau(\T^{b}_r):=f^1\otimes 1,\qquad \tau(\T^{b,b}_r):=-\varphi\otimes t,\qquad \tau(\T^{b^{\times n}}_r):=0\quad (n\ge 3).
\end{align*}
It remains to check that $\tau$ commutes with the differentials, and it is enough to check it on the generators.\ Note that $\tau$ preserves the degrees: $\psi\otimes dt$ and $m^2_B\otimes1$ have degree $1$, and $\varphi\otimes t$ has degree $0$.\ Since $m^1_A=0$ and $m^1_B=0$, the differential of $Hom(A,B)\otimes R$ is $d(\phi\otimes r)=(-1)^{\deg'(\phi)}\phi\otimes d(r)$.
\begin{itemize}
\item[(a)] \emph{Blue corollas.} $d_{A^2_\infty}(\T^{b^{\times n}}_b)$ is a sum of compositions of blue corollas, on which $\tau$ vanishes; and $d\big(\tau(\T^{b^{\times n}}_b)\big)=d(0)=0$.
\item[(b)] \emph{Red corollas.} We have $d_{A^2_\infty}(\T^{r,r}_r)=0$ and $d\big(m^2_B\otimes 1+\psi\otimes dt\big)=0$, since $d_R(1)=d_R(dt)=0$.\\ 
For $n=3$, by (\ref{eq:dblue}), $d_{A^2_\infty}(\T^{r^{\times 3}}_r)=-\T^{r,r}_r\circ_1\T^{r,r}_r-\T^{r,r}_r\circ_2\T^{r,r}_r$.\ 
By the composition rule of $Hom(B,B)\otimes R$, since $\deg'(m^2_B)=1$, $\deg'(\psi)=0$ and $dt\,dt=0$, for $k=1,2$ we have
\begin{align*}
\tau(\T^{r,r}_r)\circ_k\tau(\T^{r,r}_r)=\big(m^2_B\circ_k m^2_B\big)\otimes 1+\big(m^2_B\circ_k\psi-\psi\circ_k m^2_B\big)\otimes dt.
\end{align*}
Now $\sum_{k=1}^2 m^2_B\circ_k m^2_B=0$ is the equation (\ref{eq:Ainf-cat}) for $n=3$, and, by the formula for $\circ_k$ in $Hom(B,B)$,
\begin{align*}
\sum_{k=1}^2\big(m^2_B\circ_k\psi-\psi\circ_k m^2_B\big)(x_3,x_2,x_1)=C(x_3,x_2,x_1)=0
\end{align*}
by (\ref{eq:C}).\ So $\tau\big(d_{A^2_\infty}(\T^{r^{\times 3}}_r)\big)=0=d\big(\tau(\T^{r^{\times 3}}_r)\big)$.\\ 
For $n\ge 4$, every term of $d_{A^2_\infty}(\T^{r^{\times n}}_r)$ is a composition $\T^{r^{\times(n-j+1)}}_r\circ_k\T^{r^{\times j}}_r$ with $(n-j+1)+j=n+1\ge 5$, hence one of the two corollas has at least $3$ leaves, and $\tau$ vanishes on it.
\item[(c)] \emph{Mixed corollas.} $d_{A^2_\infty}(\T^{b}_r)=0$ and $d(f^1\otimes 1)=0$.\\ 
For $n=2$, by (\ref{eq:dmixed}), $d_{A^2_\infty}(\T^{b,b}_r)=\T^{b}_r\circ\T^{b,b}_b-\T^{r,r}_r\circ(\T^{b}_r,\T^{b}_r)$.\ Applying $\tau$, the first term vanishes, and, since $f^1$ has shifted degree $0$,
\begin{align*}
\big(m^2_B\otimes 1+\psi\otimes dt\big)\circ\big(f^1\otimes 1,f^1\otimes 1\big)=m^2_B\circ(f^1\otimes f^1)\otimes 1+\psi\circ(f^1\otimes f^1)\otimes dt=\varphi\otimes dt,
\end{align*}
because $m^2_B(e,e)=0$; hence $\tau\big(d_{A^2_\infty}(\T^{b,b}_r)\big)=-\varphi\otimes dt$.\ On the other hand $d\big(\tau(\T^{b,b}_r)\big)=-d(\varphi\otimes t)=-(-1)^{\deg'(\varphi)}\varphi\otimes dt=-\varphi\otimes dt$.\\ 
For $n\ge 3$ we have $\tau(\T^{b^{\times n}}_r)=0$, hence $d\big(\tau(\T^{b^{\times n}}_r)\big)=0$, and we have to show that $\tau\big(d_{A^2_\infty}(\T^{b^{\times n}}_r)\big)=0$.\ The differential $d_{A^2_\infty}(\T^{b^{\times n}}_r)$ is given by 
\begin{equation*}
\begin{split}
d_{A^2_\infty}\Big(\corollagen{n}{opblue}{opred}\Big)&=\sum_{j=2}^{n}\ \sum_{k=1}^{n-j+1}\corollagen{n-j+1}{opblue}{opred}\circ_{k}\corollagen{j}{opblue}{opblue}-\sum_{p=2}^{n}\ \sum_{\substack{s_1+...+s_p=n\\ s_1,...,s_p\ge 1}}\corollagen{p}{opred}{opred}\circ\Big(\corollagen{s_p}{opblue}{opred},...,\corollagen{s_1}{opblue}{opred}\Big).
\end{split}
\end{equation*}
Since $\tau$ vanishes on the blue corollas, on $\T^{r^{\times p}}_r$ for $p\ge3$ and on $\T^{b^{\times s}}_r$ for $s\ge3$, the only terms on which $\tau$ does not vanish a priori are those with $p=2$ and $(s_2,s_1)\in\{(1,2),(2,1),(2,2)\}$.\ For $n\ge 5$, if $p=2$ then $s_1+s_2=n\ge 5$, hence $s_1\ge 3$ or $s_2\ge 3$, and all the terms vanish.\ It remains to consider two cases.
\begin{itemize}
\item[$\bullet$] For $n=3$ we get, using that $f^1$ and $\varphi$ have shifted degree $0$ (so that no Koszul signs appear),
\begin{align}\label{graunz}
\tau\big(d_{A^2_{\infty}}(\T^{b^{\times3}}_r)\big)&=-\tau(\T^{r,r}_r)\circ\big(\tau(\T^{b}_r),\tau(\T^{b,b}_r)\big)-\tau(\T^{r,r}_r)\circ\big(\tau(\T^{b,b}_r),\tau(\T^{b}_r)\big)\\
&=\big(m^2_B\circ(f^1,\varphi)+m^2_B\circ(\varphi,f^1)\big)\otimes t+\big(\psi\circ(f^1,\varphi)+\psi\circ(\varphi,f^1)\big)\otimes t\,dt .\nonumber
\end{align}
Its value on $(a,a,a)$ is $\big(m^2_B(e,u)+m^2_B(u,e)\big)\otimes t+\big(\psi(e,u)+\psi(u,e)\big)\otimes t\,dt=0$, since $m^2_B$ vanishes on $V\otimes V$ and $\psi$ vanishes as soon as one of its arguments is $u$.
\item[$\bullet$] For $n=4$ we get
\begin{align}\label{graunz4}
\tau\big(d_{A^2_{\infty}}(\T^{b^{\times4}}_r)\big)=-\tau(\T^{r,r}_r)\circ\big(\tau(\T^{b,b}_r),\tau(\T^{b,b}_r)\big)=-m^2_B\circ(\varphi,\varphi)\otimes t^2-\psi\circ(\varphi,\varphi)\otimes t^2\,dt,
\end{align}
whose value on $(a,a,a,a)$ is $-m^2_B(u,u)\otimes t^2-\psi(u,u)\otimes t^2\,dt=0$.
\end{itemize}
So $\tau\big(d_{A^2_\infty}(\T^{b^{\times n}}_r)\big)=0=d\big(\tau(\T^{b^{\times n}}_r)\big)$ for every $n\ge 3$.
\end{itemize}
Finally, since $\text{ev}_k(1)=1$, $\text{ev}_k(dt)=0$, $\text{ev}_0(t)=0$ and $\text{ev}_1(t)=1$, we have $(\Id\otimes\text{ev}_0)\circ\tau=\rho_f$ and $(\Id\otimes\text{ev}_1)\circ\tau=\rho_g$.\ On the blue and red corollas both give $m_A$ and $m_B$, and on the mixed ones they give $(f^1,0,0,...)=f$ and $(f^1,-\varphi,0,...)=g$.\ Hence $f\asymp g$.
\end{proof}

\subsection*{Acknowledgements}
I thank Alberto Canonaco and Paolo Stellari for the many helpful discussions.\ 
Part of this work was written while the author was visiting the Institut des Hautes Etudes Scientifiques (Paris) 
whose warm hospitality is gratefully acknowledged.

\end{document}